\documentclass[10pt, preprint]{amsart}
\usepackage{amssymb,amscd,amsmath, amsthm, graphicx, calc, tikz, tikz-cd}
\usepackage[all]{xy}
\usepackage[color=pink,textsize=footnotesize]{todonotes}
\usepackage{color} %include even if images aren’t in color
\def\supp{\underline}
\def\sign{\mathrm{sign}}
\def\minsupp{\mathrm{Minsupp}}

\def\C{\mathbb C}
\def\F{\mathbb F}
\def\P{\mathbb P}
\def\R{\mathbb R}
\def\K{\mathbb K}
\def\S{\mathbb S}
\def\T{\mathbb T}
\def\PP{\mathbb P}

\def\TR{\mathbb T\mathbb R}
\def\TC{\mathbb T\mathbb C}

\def\phr{\varphi}
\def\TP{{\mathbb T}\P}

\def\phr{\varphi}
\def\BB{\mathcal{B}}
\def\CC{\mathcal{C}}

\def\MM{\mathcal{M}}

\def\VV{\mathcal{V}}
\def\WW{\mathcal{W}}

\def\0{\mathbf 0}

\def\ph{\operatorname{ph}}

\def\id{\mathrm{id}}

\def\min{\mathrm{min}}

\def\hplus{\boxplus}
\def\htimes{\cdot}

\newtheorem{thm}{Theorem}[section]
\newtheorem{lemma}[thm]{Lemma}
\newtheorem{prop}[thm]{Proposition}
\newtheorem{conj}[thm]{Conjecture}
\newtheorem{cor}[thm]{Corollary}

\theoremstyle{definition}
\newtheorem{defn}[thm]{Definition}
\newtheorem{nota}[thm]{Notation}
\newtheorem{remark}[thm]{Remark}
\newtheorem{example}[thm]{Example}

\newtheorem*{wcproperty}{Weak Closure Property}
\newtheorem*{eproperty}{Elimination Property}
\newtheorem*{acproperty}{Additive Continuum Property}

\def\row{\mathop{\rm Row}}

\def\mmax{\mathop{\rm max}}

\newcommand{\bighplus}{
  \mathop{
    \vphantom{\bigoplus} 
    \mathchoice
      {\vcenter{\hbox{\resizebox{\widthof{$\displaystyle\bigoplus$}}{!}{$\boxplus$}}}}
      {\vcenter{\hbox{\resizebox{\widthof{$\bigoplus$}}{!}{$\boxplus$}}}}
      {\vcenter{\hbox{\resizebox{\widthof{$\scriptstyle\oplus$}}{!}{$\boxplus$}}}}
      {\vcenter{\hbox{\resizebox{\widthof{$\scriptscriptstyle\oplus$}}{!}{$\boxplus$}}}}
  }\displaylimits 
}

\begin{document}
\title{Vectors of matroids over tracts}
\author{Laura Anderson}
\email{laura@math.binghamton.edu}
\address{Department of Mathematical Sciences, Binghamton University, Binghamton, NY 13902-6000, USA.}
\begin{abstract} We enrich Baker and Bowler's theory of matroids over tracts with notions of vectors and covectors. In the case of oriented matroids, these $F$-vectors and $F$-covectors coincide with the usual signed vectors and signed covectors, and in the case of ordinary matroids, they are essentially the unions of circuits resp. unions of cocircuits.. In the case of matroids over a field $F$, the $F$-covector set resp.\ $F$-vector set of an $F$-matroid is a linear subspace of $F^E$ resp.\ its orthogonal complement.
\end{abstract}

\maketitle

The theory of ``matroids with extra structure", including oriented matroids (\cite{BLSWZ}), valuated matroids (\cite{DW92}), and phased matroids (originally called ``complex matroids" in~\cite{AD}),\footnote{Since publishing~\cite{AD}, my coauthor and I have come to agree with commenters that the term ``complex matroids" used there was a poor choice, since it suggests matroids realizable over the complex numbers.} recently found a beautiful common description in a paper of Baker and Bowler (\cite{BB17}, see also \cite{BB16}) as {\em matroids over hyperfields}, and more generally, {\em matroids over tracts}. A hyperfield is similar to a field, but with  addition being multivalued. A tract is a generalization of hyperfield that also encompasses partial fields and fuzzy rings. The concept of matroids over a tract generalizes the concept of linear subspaces of a vector space $F^n$ -- in fact, when $F$ is a field then an $F$-matroid corresponds exactly to a linear subspace of $F^n$. 

One example of a hyperfield is the {\em sign hyperfield }  $\S$ (Definition~\ref{def:signhyp}). Matroids over $\S$ are exactly oriented matroids. The theory of oriented matroids is the grandmother of all theories of matroids with extra structure, with a rich theory and powerful connections to geometry and topology. Baker and Bowler's paper suggests broad generalizations of this theory across other tracts. However, two quirks of their work stand out:
\begin{enumerate}
\item they did not generalize one key axiomatization of oriented matroids, namely signed vector axioms and signed covector axioms, and
\item for a general tract $F$, there are distinct notions of {\em strong matroids over $F$} and {\em weak matroids over $F$}. When $F$ is a field or when $F=\S$, these two notions coincide, but for general tracts it was not clear which notion deserved greater prominence.
\end{enumerate}

The  present paper fills the gap (1) and weighs in on (2). We give a definition of $F$-vectors and $F$-covectors of a strong matroid over a tract $F$ so that:
\begin{enumerate}
\item if $F$ is a field, so that a (strong) $F$-matroid corresponds to a subspace of some $F^E$, then the $F$-covectors of that $F$-matroid are just the elements of the subspace, 
\item if $F=\S$, so that a (strong) $F$-matroid is an oriented matroid, then the $F$-vectors and $F$-covectors  coincide with the signed vectors and signed covectors in the sense of oriented matroids, and
\item for general tracts $F$, the relationships between $F$-(co)vectors and $F$-(co)circuits for strong $F$-matroids is similar to the relationships for oriented matroids.
\end{enumerate}
However, our definitions look substantially different from the usual oriented matroid definitions. Also, for weak $F$-matroids our $F$-vectors are not cryptomorphic to the other $F$-matroid axiom systems.  Thus the present work can be considered an argument in favor of the strong notion of $F$-matroids.

 Section 1 will give background, including a summary of some results on matroids over tracts from~\cite{BB17}  and examples of hyperfields of particular interest, introduced by Viro in~\cite{Viro1}.  Section 2 will introduce vector and covector  axioms for $F$-matroids and will state the most fundamental results --  in particular, cryptomorphisms  between these new axiomatizations  and those in~\cite{BB17} for $F$-matroids (Theorem~\ref{thm:cryptom}).  These results are proved in Section 3.  Section 4 discusses how various  matroid properties --  such as duality, deletion, and contraction --  play out for vectors and covectors. This section has some distressing surprises: several properties that we use often in the context of oriented matroids and of subspaces of vector spaces fail to hold for general $F$-matroids.  Section 5  explores some specific examples.
 
Sections 6 and 7 address the discrepancy between the vector axioms for $F$-matroids and those for oriented matroids. The Composition Axiom does not hold for general $F$-matroids: Section 6 discusses the potential role of operations similar to composition with respect to various tracts, including many of those introduced by Viro.  More surprisingly, the Elimination Axiom also does not hold for general $F$-matroids. For matroids over a particular tract, one can conjecture both weaker and stronger variations on Elimination of obvious interest, and for many $F$  it is unknown whether any of these variations hold.
\vskip 12pt

\section{Background}\label{sec:background}

This section will very briefly review definitions from Baker and Bowler's paper (\cite{BB17}), but only as a reminder to the reader who has already read the more complete treatment there. 

\begin{nota}  $E$ denotes a finite set. $[n]$ denotes $\{1,2,\ldots,n\}$. $F$ denotes a set $G\cup\{0\}$, where $(G, N_G)$ is a tract (Definition~\ref{def-tract}). %$F^\times$ denotes $F-\{0\}$. 

$\0$ denotes  the vector in $F^E$ with all components $0$.

The {\bf support} of $X\in F^E$ is $\supp{X}:=\{e\in E: X(e)\neq 0\}$. 
%If $A\subseteq F^E$, we set $\supp{A}:=\{\supp{X}: X\in A\}$.
 The {\bf zero set} of $X$ is $X^0:=\{e\in E: X(e)= 0\}$.  

The letter $i$ will always denote $\sqrt{-1}$ in $\mathbb C$.

If $X\in F^E$ and $e\in E$ then $X\backslash e$ denotes the restriction of $X$ to $E-\{e\}$. If $f\not\in E$ and $\alpha\in F$ then $Xf^\alpha$ denotes the extension of $X$ to $E\cup\{f\}$ with $X(f)=\alpha$.

For any $S\subseteq F^E$, $\minsupp(S)$ denotes the set of elements of $S$ of minimal support.

\end{nota}

\subsection{Tracts and hyperfields}

\begin{defn}\label{def-tract} (\cite{BB17}) A {\bf tract} is a multiplicative abelian group $G$ together with a subset $N_G$ of the group semiring ${\mathbb N}[G]$ satisfying all of the following.
\begin{enumerate}
\item the zero element of ${\mathbb N}[G]$ belongs to $N_G$.
\item the identity element $1$ of $G$ is not in $N_G$.
\item There is a unique element $\eta$ of $G$ with $1+\eta\in N_G$.
\item $N_G$ is closed under the  action of $G$ on $N_G$.
\end{enumerate}
\end{defn}

We often refer to the set $F=G\cup\{0\}$ as the tract, and for $g\in G$ we often denote $\eta g$ as $-g$. If $a_1, \ldots, a_k\in F$ then $a_1\hplus\cdots \hplus a_k$ (or $\bighplus_{j=1}^k a_j$) denotes $\{b\in F: -b+\sum\limits_{j=1}^k a_j\in N_G\}$. The canonical example of a tract is when $F$ is a field, $G=F-\{0\}$, and $N_G$ is the set of all formal sums of elements of $F$ that add to 0 in $F$. In this case the set $\bighplus_{j=1}^k a_j$ has exactly one element, the actual sum of the $a_j$. 

A particularly interesting class of tracts arise from {\em hyperfields}, which we now define.

\begin{defn} A {\bf hyperoperation} on a set $R$ is a map $\hplus$ from $R\times R$ to the set of nonempty subsets of $R$.

If $S$ and $T$ are subsets of $R$ then we define $S\hplus T$ to be $\bigcup_{\substack{s\in S\\ t\in T}}s\hplus t$, and for any $x\in R$ we define $x\hplus S=\{x\}\hplus S$ and $S\hplus x=S\hplus\{x\}$.
\end{defn}

Thus, for instance, if $a,b,c\in R$ then $(a\hplus b)\hplus c$ is defined to be ${\bigcup_{x\in a\hplus b} x\hplus c}$.

As in~\cite{BB17}, all hyperoperations in this paper will be commutative and associative, with identity $0$. Thus for any finite $S\subseteq R$ the sum ${\bighplus_{s\in S} s}$ is well-defined, with ${\bighplus_{s\in \emptyset} s}$ defined to be $\{0\}$.

Note that the $\hplus$ defined for a tract is not necessarily a hyperoperation, since for elements $a,b$ in a tract, $a\hplus b$ might be the empty set.

\begin{defn} [\cite{BB17})]\label{defn:hypergroup} A commutative {\bf hypergroup} is a tuple $(G,\hplus,0)$, where $\hplus$ is a commutative and associative hyperoperation on $G$, such that
\begin{enumerate}
\item  $0\hplus x=\{x\}$ for all $x\in G$.
\item  For every $x\in G$ there is a unique element of $G$, denoted $-x$, such that $0\in x\hplus -x$.
\item  For all $x$, $y$, and $z$, $x\in y\hplus z$ if and only if $z\in x\hplus (-y)$.
\end{enumerate}

A commutative {\bf hyperfield} is a tuple $(R,\htimes,\hplus,1,0)$ such that $0\neq 1$ and
\begin{enumerate}
\item $(R-\{0\},\htimes, 1)$ is a commutative group.
\item $(R, \hplus, 0)$ is a commutative hypergroup.
\item $0\htimes x=x\htimes 0=0$ for all $x\in R$.
\item $a\htimes(x\hplus y)=(a\htimes x)\hplus(a\htimes y)$ for all $a,x,y\in R$.
\end{enumerate}

\end{defn}

If $F$ is a hyperfield then let $G=F-\{0\}$ and let $N_G=\{\sum\limits_{j=1}^k a_j: 0\in\bighplus_{j=1}^k a_j\}$. Then $G$ and $N_G$ define a tract whose associated operation $\hplus$ coincides with the hyperaddition in $F$. 

If $F$ is a tract and $X,Y\in F^E$ then we define $$X\hplus Y:=\{Z: \forall e\   Z(e)\in X(e)\hplus Y(e)\}.$$ $F$ acts on $F^E$ by componentwise multiplication. (If $F$ is a hyperfield then $F^E$ is an $F$-module in the sense of~\cite{BB17}.)

\begin{defn} [\cite{BB17}] A homomorphism $(G, N_G)\to (G', N_{G'})$ of tracts is a group homomorphism $f:G\to G'$ together with a map $\tilde{f}:{\mathbb N}[G]\to{\mathbb N}[G']$ such that $\tilde{f}(\sum a_jg_j)=\sum a_jf(g_j)$ for all $a_j\in\mathbb N$ and $g_j\in G$ and $F(N_G)\subseteq N_{G'}$.
\end{defn}

We'll also use $f$ to denote the extension $(F=G\cup\{0\})\to (F'=G'\cup\{0\})$ sending 0 to 0, and we'll use $f$ to denote the componentwise map $F^E\to (F')^E$.

If $F$ and $F'$ are hyperfields then a homomorphism of the corresponding tracts amounts to a function $f: F\to F'$ such that $f(xy)=f(x)f(y)$ and $f(x\hplus y)\subseteq f(x)\hplus f(y)$ for all $x$ and $y$.

\begin{defn} [\cite{BB17}] Let $F=G\cup\{0\}$ be a tract equipped with an automorphism $c$ such that $c^2=\id$, which we call {\bf conjugation}. Denote the image of $x\in F$ under $c$ by $x^c$.  For $X,Y\in F^E$, the {\bf inner product} is defined as
$$X\htimes Y:=\sum_{e\in E} X(e)\htimes Y(e)^c\in {\mathbb N}[G].$$
We say $X$ and $Y$ are {\bf orthogonal}, denoted $X\perp Y$, if $X\htimes Y\in N_G$. If $S\subseteq F^E$, we denote by $S^\perp$ the set of all $X\in F^E$ such that $X\perp Y$ for all $Y\in S$.
\end{defn}

We will always view a tract as being equipped with $c$. When no $c$ is specified, we will take $c=\id$.

\begin{example} \label{ex:tracts}
Viro's paper~\cite{Viro1}  provides an excellent introduction to and motivation for hyperfields. Several of the following hyperfields were first introduced there.

Notice that several of these hyperfields have operations defined in terms of the usual field operations on $\R$ and $\C$. The symbol $+$, when denoting a binary operation, always denotes the usual addition.
\begin{enumerate}
\item The field $\R$ is a hyperfield, with $c=\id$.
\item The field $\C$ is a hyperfield, with $c$ being conjugation.

\item The {\bf Krasner hyperfield} $\K$ on elements $\{0,1\}$ is the unique two-element hyperfield for which $1\hplus 1=\{0, 1\}$. The automorphism $c$ is the identity.
\item \label{def:signhyp} The {\bf sign hyperfield} $\S$ on elements $\{0,+,-\}$ has addition $0\hplus x=\{x\}$ for all $x$, $x\hplus x=\{x\}$ for all $x$, and $+\hplus -=\{0,+,-\}$. Multiplication is given by $0\odot x=0$ for all $x$, $+\odot +=-\odot -=+$, and $+\odot -=-$. The automorphism $c$ is the identity.
\item \label{phasetracts} Let $S^1$ denote the unit circle in $\C$. For $x\in\C$, the {\bf phase} of $x$ is 
$$\ph(x)=\begin{cases}
0&\mbox{ if $x=0$}\\
\frac{x}{|x|}&\mbox{ otherwise}
\end{cases}
$$
The {\bf phase hyperfield} $\P$ on elements $S^1\cup\{0\}$ has addition $x\hplus y:=\{\ph(ax+by): a,b\in\R_{>0}\}$ and has multiplication inherited from $\C$. The automorphism $c$ is conjugation. (For nonzero $x$ we have $x^c=x^{-1}$, and in~\cite{AD} the inner product is described in this way.)

This is not the hyperfield that gets the name ``phase hyperfield" in~\cite{Viro1}: see~(\ref{tropicalphase}) below.

\item (\cite{Viro1}) The {\bf triangle hyperfield} $\triangle$ on elements $\R_{\geq 0}$ has addition $x\hplus y:=\{z: |x-y|\leq z\leq x+y\}$ and has multiplication inherited from $\R$. The automorphism $c$ is the identity.

\item  (\cite{Viro1}) The {\bf tropical real hyperfield} $\TR$ on elements $\R$ has addition 
$$x\hplus y:=\begin{cases}
x&\mbox{ if $|x|>|y|$ or $x=y$}\\
y&\mbox{ if $|x|<|y|$}\\
\{z: |z|\leq |x|\}&\mbox{ if $x=-y$}
\end{cases}$$
and has multiplication inherited from $\R$. The automorphism $c$ is the identity.

\item  (\cite{Viro1}) The {\bf tropical complex hyperfield}  $\TC$ on elements $\C$ has addition
$$x\hplus y:=\begin{cases}
x&\mbox{ if $|x|>|y|$}\\
y&\mbox{ if $|x|<|y|$}\\
\{|x|\ph(ax+by):a,b\in\R_{\geq 0}\}&\mbox{ if $|x|=|y|$ and $x\neq -y$}\\
\{z: |z|\leq |x|\}&\mbox{ if $x=-y$}
\end{cases}$$
and has multiplication inherited from $\C$. The automorphism $c$ is conjugation.

\item \label{tropicalphase} (\cite{Viro1}) The {\bf tropical phase hyperfield} $\TP$ on elements $S^1\cup\{0\}$ (called the {\em phase hyperfield} in~\cite{Viro1}) has addition 
$$x\hplus y:=\begin{cases}
S^1\cup\{0\}&\mbox{ if $x=-y\neq 0$}\\
\{\ph(ax+by): a,b\in\R_{\geq0}\}&\mbox{otherwise}
\end{cases}$$
 and has multiplication inherited from $\C$.  The automorphism $c$ is conjugation.

\item  (\cite{Viro1}) The {\bf ultratriangle hyperfield} $\T\triangle$  on elements $\R_{\geq 0}$ has addition
$$x\hplus y=\begin{cases}
\max(x,y)&\mbox{ if $x\neq y$}\\
\{z:z\leq x\}&\mbox{ if $x=y$}
\end{cases}$$
and has multiplication inherited from $\R$.  The automorphism $c$ is the identity.
\end{enumerate}
\end{example}

\begin{example} Here are some examples of conjugation-preserving morphisms of tracts.

 \begin{enumerate}
 \item The inclusion $\R\to\C$ is a morphism of tracts.
\item For every tract $F$ the function $\kappa:F\to \K$ sending each nonzero element of $F$ to $ 1$ is a morphism of tracts. 
\item The function $\ph: \C\to\P$ introduced in Example~\ref{ex:tracts}.\ref{phasetracts} is a morphism of tracts. It restricts to a morphism $\sign: \R\to\S$.
\item Other morphisms which are interesting but will not be considered further here include 
\begin{enumerate} 
\item  $\C\to\triangle$ taking each $x\in\C$ to $|x|$
\item $\sign:\TR\to \S$ 
\item $\ph: \TC\to\TP$
\item $\T\C\to\T\triangle$ taking each $x\in\T\C$ to $|x|$
\end{enumerate}
\end{enumerate}
\end{example}

\subsection{Matroids over tracts}

\begin{defn}\cite{BB17} Let $L$ be a finite lattice with minimal element $\hat{0}$. An element $a\in L$ is an {\bf atom} if $a\neq\hat{0}$ and there is no $b\in L$ such that $\hat{0}<b<a$. A set $S$ of atoms in $L$ is a {\bf modular family} if the height of $\bigvee S$ in $L$ is $|S|$.

If $\CC$ is a subset of $F^E$ then  $S\subseteq\CC$ is called a {\bf modular family} if the set of supports of elements of $S$ is a modular family in the lattice of unions of supports of elements of $\CC$.
\end{defn} 

\begin{defn}[\cite{BB17}] \label{strongcircaxs}  Let $E$ be a nonempty finite set, $F$  a tract, and $\CC\subseteq F^E$.  We say $\CC$ is  the set of {\bf  $F$-circuits of a strong $F$-matroid on $E$} if $\CC$ satisfies all of the following axioms.
\begin{description}
\item [Nontriviality] $\0\not\in \CC$.
\item [Symmetry] If $X\in \CC$ and $\alpha\in F-\{0\}$ then $\alpha X\in \CC$.
\item [Incomparability] If $X,Y\in \CC$ and $\supp{X}\subseteq\supp{Y}$ then there exists $\alpha\in F-\{0\}$ such that $X=\alpha Y$.
\item [Strong modular elimination] Let $\{X_1, \ldots, X_k, X\}\subseteq\CC$ be a modular family of size $k+1$ such that $\supp{X}\not\subseteq\cup_{j=1}^k\supp{X_j}$, and for each $1\leq j\leq k$ let $e_j\in (\supp{X}\cap \supp{X_j})\backslash\cup_{l\neq j} \supp{X_l}$ be such that $X(e_j)=-X_j(e_j)$. Then there is an $F$-circuit $Z\in\CC$ such that $Z(e_j)=0$ for each $1\leq j\leq k$ and $Z\in X\hplus \bighplus_{j=1}^k X_j$.
\end{description}
\end{defn}

\begin{remark} Baker and Bowler also defined {\em weak $F$-matroids}, which only require Strong Modular Elimination for modular families of size 2. 

If $F$ is a field, $F=\K$, or $F=\S$, then every weak $F$-matroid is a strong $F$-matroid. If $F$ is a tract for which strong and weak $F$-matroids coincide, we will refer simply to {\bf $F$-matroids}.
\end{remark}

The following is part of Theorem 2.24 in~\cite{BB17}.
\begin{thm} Let $\CC$ be the $F$-circuit set of a strong $F$-matroid. Then $\CC^*:=\minsupp(\CC^\perp-\{\0\})$ is also the $F$-circuit set of a strong  $F$-matroid, and $(\CC^*)^*=\CC$.
\end{thm}

A {\bf  strong $F$-matroid}  is the information encoded in a set $\CC$ satisfying Definition~\ref{strongcircaxs}. Thus a strong $F$-matroid $\MM$ has an associated pair $(\CC(\MM),\CC^*(\MM))$. $\CC(M)$ is the set of {\bf $F$-circuits of $\MM$}, and $\CC^*(\MM)$   is the set of {\bf $F$-cocircuits of $\MM$}. The strong $F$-matroid with $F$-circuit set $\CC^*(\MM)$ and $F$-cocircuit set $\CC(\MM)$ is called the {\bf dual} to $\MM$ and is denoted $\MM^*$.

\begin{remark} Baker and Bowler give another characterization of $F$-matroids as well, via {\em Grassmann-Pl\"ucker functions}, which generalize chirotopes of oriented matroids.
\end{remark}

\begin{example} Matroids in the usual sense are essentially $\K$-matroids. Specifically, if $S\subseteq E$ let $X_S\in \K^E$ be the indicator function for $S$ (so $X_S(e)= 1$ if and only if $e\in S$). Then $C\subseteq 2^E$ is the set of circuits of a matroid in the usual sense if and only if $\{X_S: S\in C\}$ is the set of $\K$-circuits of a $\K$-matroid (Corollary 1 in~\cite{Del11}).

When we refer to a {\em matroid} without modification, we mean a matroid in the traditional sense -- an object with circuits which are subsets of $E$. When we wish to use the language of matroids over tracts, we will talk about $\K$-matroids.
\end{example}

\begin{example} $\CC\subseteq\S^E$ is the set of $\S$-circuits of an $\S$-matroid if and only if $\CC$ is the set of signed circuits of an oriented matroid. This follows from Corollary 1 in~\cite{Del11} together with Theorem 3.6.1 in~\cite{BLSWZ}.
\end{example}

\begin{example} If $K$ is a field, 
then $\CC^*\subseteq K^E$ satisfies the $K$-circuit axioms if and only if $\CC^*=\minsupp(V-\{\0\})$ for some linear subspace $V$ of $K^E$ (\cite{BB17}). We call the $K$-matroid with $K$-cocircuit set $\CC^*$ the {\bf $K$-matroid corresponding to $V$.}
\end{example}

\begin{prop}[\cite{BB17}] \label{prop:morph} A conjugation-preserving morphism $f:F\to F'$ of tracts induces a map $f_*$ from (weak resp.\ strong) $F$-matroids to  (weak resp.\ strong) $F'$-matroids so that 
$$\CC(f_*(\MM))=\{\alpha f(X): X\in\CC(\MM), \alpha\in F'-\{0\}\}$$
and 
$$\CC^*(f_*(\MM))=\{\alpha f(X): X\in\CC^*(\MM), \alpha\in F'-\{0\}\}.$$
\end{prop}

\begin{example} If $\MM$ is an $F$-matroid then the ordinary matroid corresponding to $\kappa_*(\MM)$ is called the {\bf  underlying matroid} of $\MM$.

Since underlying matroids are well-defined, we can refer to classical matroid properties of an $F$-matroid, e.g.\  bases and circuits. (For instance, a circuit of an $F$-matroid $\MM$ is defined to be a circuit of the underlying matroid, hence is the support of an $F$-circuit of $\MM$.) 
\end{example}

\begin{example} If $V$ is a linear subspace of $\R^n$ and $\MM_V$ is the $\R$-matroid associated to $V$ then $\sign_*(\MM)$  is the oriented matroid associated to $V$.
 Similar comments hold for subspaces of $\C^n$ and strong $\P$-matroids.
\end{example}

\begin{cor}\label{cor:underlying} If $\MM$ is a (weak or strong) $F$-matroid then there is a bijection $G X\to\supp{X}$ from the set of $G$-orbits of circuits of $\MM$ to the set of circuits of the underlying matroid.
\end{cor}

\section{Main Theorem: $F$-vectors}

This section introduces our new axiomatization of strong $F$-matroids.
The idea is to generalize the following vision of subspaces of a vector space: let $K$ be a field and $V$ the row space of a matrix $M$ over $K$ with columns indexed by $E$. A subset $B$ of $E$ is a basis of the $K$-matroid corresponding to $V$ if and only if the set of columns indexed by $B$ is a basis for the column space. In this case $M$ is equivalent by row operations to a matrix $M_B$ such that the submatrix of $M_B$ with columns indexed by $B$
is an identity matrix. We call $M_B$ a {\em reduced row-echelon form for $M$ with respect to $B$}. A necessary 
condition for an element $X$ of $K^E$ to be in $V$ is that $X$ must be a linear combination of the rows of $M_B$. (Of course, because $K$ is a field this is also a sufficient condition, but for general tracts, and even hyperfields, this will take more thought.)

Let $\{R_j:j\in B\}$ denote the set of rows of $M_B$, where $R_j$ is the unique element of $V$ such that for all $k\in B$, $R_j(k)=\delta_{jk}$.  (Here 
 $\delta_{jk}$ denotes the Kronecker delta.)
The above condition for $X$ to be in $V$ can be rephrased as
\begin{equation}\label{eqn:rref}
X=\sum_{j\in B} X(j)R_j
\end{equation}

Our generalization of rows in reduced-row-echelon-form matrices is $F$-cocircuits, and our definition of $F$-covectors is vectors satisfying Equation~(\ref{eqn:rref}) (with hypersum $\boxplus$ in place of sum) with respect to every basis. 

In order to get a stand-alone definition of $F$-vectors and $F$-covectors, we first need to define bases without reference to circuits or cocircuits.

\begin{defn} Let $F$ be a tract, $E$ a finite set, and $\WW\subseteq F^E$. 
A {\bf support basis} of $\WW$ is a minimal $B\subseteq E$ such that $B\cap\supp{X}\neq\emptyset$ for every $X\in\WW-\{\0\}$.
\end{defn}

\begin{lemma}\label{lem:supportbasis1}  1. For any $\WW\subseteq F^E$, the set of support bases of $\WW$ is the set of support bases of $\minsupp(\WW-\{\0\})$.

2. If $\MM$ is a (weak or strong) $F$-matroid on $E$ and $B\subseteq E$ then $B$ is a basis for $\MM$ if and only if $B$ is a support basis for $\CC^*(\MM)$.
\end{lemma}

\begin{proof} (1) is clear. (2) follows from Corollary~\ref{cor:underlying} together with Proposition 2.1.16 in~\cite{Oxley}.
\end{proof}

\begin{defn} Let $B$ be a support basis for $\WW$. 
A {\bf nearly reduced row-echelon form} for $\WW$ with respect to $B$ is a subset $\{S_j:j\in B\}$ of $\WW$ such that $\supp{S_j}\cap B=\{j\}$ for each $j\in B$. A {\bf reduced row-echelon form} for $\WW$ with respect to $B$ is a nearly reduced row-echelon form $\{R_j:j\in B\}$ such that $R_j(k)=\delta_{jk}$ for each $j,k\in B$.
\end{defn}

In other words, a reduced row-echelon form is a nearly reduced row-echelon form $\{S_j:j\in B\}$ such that $S_j(j)=1$ for every $j$. A reduced row-echelon form with respect to $B=\{b_1, \ldots, b_r\}$ is a subset $\{R_{b_j}:j\in[r]\}$ of $\WW$ such that, if $(b_1, \ldots, b_r)$ is extended to an ordering $(b_1, \ldots, b_n)$ of $E$, then $\{R_{b_j}:j\in[r]\}$ of $\WW$ is the set of rows of a matrix with columns indexed by $E$ of the form $(I|A)$.

\begin{lemma}\label{lem:minrref} Let $\WW\subseteq F^E$ and $B$ a  support basis.
\begin{enumerate}
\item  $\WW$ has at least one nearly reduced row-echelon form with respect to $B$.
\item If $\WW$ is closed under  multiplication by $F-\{0\}$ then $\WW$ has at least one reduced row-echelon form with respect to $B$. 
\end{enumerate}
\end{lemma}

The proof is straightforward.

\begin{lemma}\label{lem:circrref} If $\MM$ is a (weak or strong) $F$-matroid then, for every basis $B$, $\CC^*(\MM)$ has a unique reduced row-echelon form with respect to $B$.
\end{lemma}

This follows  from the following lemma, which introduces the idea of {\bf fundamental $F$-circuit $FC(e,B)$} and  {\bf fundamental $F$-cocircuit $FC^*(e,B)$}. The reduced row-echelon form promised by Lemma~\ref{lem:circrref} is $\{FC^*(j,B):j\in B\}$.

\begin{lemma}\label{lem:fundcirc} Let $\MM$ be a (weak or strong) $F$-matroid and $B$ a basis for $\MM$.
 \begin{enumerate}
\item  If $e\in E-B$ then there is a unique element of $\CC(\MM)$, denoted $FC(e,B)$, such that $\supp{FC(e,B)}\subseteq B\cup\{e\}$ and $FC(e,B)(e)=1$.
\item If $j\in B$ then there is a unique element of $\CC^*(\MM)$, denoted $FC^*(j,B)$, such that $\supp{FC^*(j,B)}\subseteq (E-B)\cup\{j\}$ and $FC^*(j,B)(e)=1$.
\item If $X\in\CC(\MM)$ and $e\in\supp{X}$ then there is a basis $B'$ for $\MM$ such that $X$ is a scalar multiple of $FC(e,B')$.
\end{enumerate}
\end{lemma}

\begin{proof} 1. An elementary result of matroid theory (cf. 1.2.6 in~\cite{Oxley}) says that
 there is a unique circuit $C(e,B)$ of the underlying matroid of $\MM$ such that $e\in C(e,B)\subseteq B\cup\{e\}$.
  By Proposition~\ref{prop:morph} there is an element of $\CC(\MM)$ with support $C(e,B)$, and by Symmetry and Incomparability there is a unique such element with value 1 on $e$.
  
 The proof of (2) is similar.
 
 The proof of (3) follows from a standard  matroid theory result: $\supp{X}-\{e\}$ is independent in $\MM$, hence extends to a basis $B'$. Uniqueness of $C(e,B')$ implies that $\supp{X}=C(e,B')$, and so the result follows from Corollary~\ref{cor:underlying}.
  \end{proof}

\begin{defn} 
Let $X, Y_1, \ldots, Y_k\in F^E$. We say $X$ is a {\bf linear combination of $\{Y_1, \ldots, Y_k\}$} if there are $\alpha_1,\ldots,\alpha_k\in F$ such that $X\in\bighplus_{j=1}^k \alpha_j Y_j$.
\end{defn}

The following is clear but is important enough to be a lemma:
\begin{lemma}\label{lem:coeffs} If  $B\subseteq E$ and $\{R_j: j\in B\}\subseteq F^E$ satisfies $R_j(k)=\delta_{jk}$ for all $k\in B$ then $X\in F^E$ is a linear combination of $\{R_j: j\in B\}$ if and only if, for each $e\in E$,
$$X(e)\in\bighplus_{j\in B} X(j)R_j(e)=\bighplus_{j\in B\cap\supp{X}} X(j)R_j(e).$$
\end{lemma}

We now arrive at the axiomatization of $F$-vectors of $F$-matroids.
\begin{defn} [\bf Tract Vector Axiom, First Form] \label{def:vectorax} Let $F$ be a tract, $E$ a finite set, and $\WW\subseteq F^E$.  $\WW$ is an {\bf $F$-vector set} if $\WW$ is exactly the set of all $X\in F^E$ such that, for every support basis $B$ and every nearly reduced row echelon form $\{S_j:j\in B\}$ with respect to $B$, $X$ is a linear combination of $\{S_j:j\in B\}$.
 \end{defn}

We make this the definition for the sake of brevity, but it is easy to verify the following equivalent characterization:
\begin{prop}[\bf Tract Vector Axiom, Second Form] 
 Let $F$ be a tract, $E$ a finite set, and $\WW\subseteq F^E$.  $\WW$ is an $F$-vector set if 
and only if
\begin{enumerate}
\item $\WW$ has a reduced row-echelon form with respect to each support basis $B$, and
\item $\WW$ is exactly the set of  $X\in F^E$ such that $X$ is a  linear combination of each reduced row-echelon form for $\WW$.
\end{enumerate}
\end{prop}

\begin{prop} \label{prop:urref} When $\WW$ is an $F$-vector set and $B$ is a support basis then the reduced row-echelon form $\{R_j:j\in B\}$ is  unique. 
\end{prop}

\begin{proof} If $S_j\in\WW$ satisfies $S_j(k)=\delta_{jk}$ for all $k\in B$, then by Lemma~\ref{lem:coeffs}, for every $e$,
$S_j(e)\in\bighplus_{k\in B\cap\supp{S_j}} S_j(k)R_k(e)=\{R_j(e)\}$. Thus $S_j=R_j$.
\end{proof}

For any subset $\mathcal S$ of $F^E$ for which reduced-row-echelon forms are unique, we will use  $\{R_j:j\in B\}$ to denote the reduced row-echelon form for $\mathcal S$ with respect to a basis $B$. 

\begin{example} Corollary~\ref{cor:OM} will show that, in the case $F=\S$, Definition~\ref{def:vectorax} coincides with the standard definition of signed vectors of an oriented matroid (cf.\ 3.7.5 in~\cite{BLSWZ}). That is, $\VV\subseteq \mathcal S^E$ is an $\S$-vector set if and only if $\0\in\VV$ and $\VV$ satisfies {\em Symmetry, Composition,} and {\em Elimination}.

For general $F$, if $\WW\subseteq F^E$ satisfies the tract vector axioms then $\0\in\WW$ and $\WW$ satisfies an analog to the Symmetry Axiom for oriented matroids (Lemma~\ref{lem:symmetry}). In general $\WW$ need not satisfy the Elimination or Composition Axioms (see Sections~\ref{sec:elim} and~\ref{sec:flatcompos}). 
\end{example}

We now begin to relate Definition~\ref{def:vectorax} to $F$-matroids.

\begin{lemma}\label{lem:suppbasesarebases} Let $\WW\subseteq F^E$ satisfy the Tract Vector Axiom, and let $\BB$ be the set of support bases of $\WW$. Then $\BB$ is the set of bases of a matroid. A set $\{e_1, \ldots, e_k\}\subseteq E$ is independent in this matroid if and only if for each $j\in[k]$ there is a $Y_j\in\WW$ such that $\supp{Y_j}\cap\{e_1,\ldots, e_k\}=\{e_j\}$.
\end{lemma}

\begin{proof} We first prove that if $B$ is a support basis, $e\in E-B$, and $Y(e)\neq 0$ for some $Y\in\WW$ then there is an $f\in B$ such that $(B-\{f\})\cup\{e\}$ is a support basis. Since $Y$ is a  linear combination
of  $\{R_j: j\in B\}$ and  $Y(e)\neq 0$, there is some $f\in B$ such that $R_f(e)\neq 0$. Fix such an $f$.

If $Z$ is an element of $\WW$ and $\supp{Z}\cap (B-\{f\})=\emptyset$, then since $Z$ is a linear combination of $\{R_j: j\in B\}$ we have that $Z$ is a multiple of $R_f$. Thus $e\in\supp{Z}$. Thus $(B-\{f\})\cup\{e\}$ contains a support basis $B'$. We will see that $B'=(B-\{f\})\cup\{e\}$. To see this, let $\{R_j':j\in B'\}$ be the reduced row-echelon form with respect to $B'$. Since $R_f\cap(B'-\{e\})=\emptyset$, by the same reasoning as before we see that $R_e'$ is a multiple of $R_f$. Also by this reasoning, we see that if $Z$ is an element of $\WW$ and $\supp{Z}\cap (B'-\{e\})=\emptyset$, then $F$ is a multiple of $R'_e$, and thus is a multiple of $R_f$. Thus $(B'-\{e\}\cup\{f\}$ contains a support basis. This support basis is contained in the support basis $B$, and so it must equal $B$. Thus $B'=(B-\{f\})\cup\{e\}$.

Thus the support bases of $\WW$ satisfy the Basis Exchange Axiom for matroids: if $B$ is a support basis and  $e\in E-B$ is in a basis, then certainly there is a $Y\in\WW$ with $Y(e)\neq 0$, and so there is an $f\in B$ such that $(B-\{f\})\cup\{e\}$ is a support basis.

To see the second statement, first note that if $\{e_1,\ldots, e_k\}$ is independent, then it is contained in a basis, which implies the existence of the $Y_j$. To see the converse, we induct on $k$. Since $Y_j(e_j)\neq 0$, by the above argument each $e_j$ is in a basis. If $B$ is a basis containing $\{e_1, \ldots, e_{k-1}\}$, then consider the reduced row-echelon form $\{R_j:j\in B\}$. 
Then $Y_k=\bighplus_{j\in B} \alpha_j R_j$, for some values $\alpha_j$, but since $Y_k(e_l)=0$ for all $l<k$, we have $\alpha_{e_l}=0$ for all $l<k$. Thus there is some $j\in B-\{e_1, \ldots, e_{k-1}\}$ such that $R_j(e_k)\neq 0$, and the above argument shows that $(B-\{j\})\cup\{e_k\}$ is a basis.
\end{proof}

The following is obvious, but useful enough to state as a lemma.
\begin{lemma} \label{lem:orbitreps} Let $\CC\subset F^E$, and let $S\subseteq \CC$ contain at least one representative from each $G$ orbit in $\CC$. Then $\CC^\perp=S^\perp$.
\end{lemma}

In particular (see Lemma~\ref{lem:fundcirc}),   for an $F$-matroid $\MM$,
to show that an $X\in F^E$ is in $\CC(\MM)^\perp$, it is enough to show that $X\perp FC(e,B)$ for every $B$ and $e$.

\begin{defn} Let $\MM$ be a (strong or weak) $F$-matroid and $X\in F^E$. We say $X$ is {\bf consistent} with $\CC^*(\MM)$ if, for each basis $B$, $X$ is a linear combination of the reduced row echelon form $\{R_j:j\in B\}$ for $C^*(\MM)$ with respect to $B$.
\end{defn}

\begin{lemma} \label{lem:cstarcperp} Let $\MM$ be a (strong or weak) $F$-matroid and $X\in F^E$. Then $X\in\CC(\MM)^\perp$ if and only if $X$ is consistent with $\CC^*(\MM)$.
\end{lemma}

\begin{proof}  We will show, for every basis $B$ and every $e\in E-B$, that $X\perp FC(e,B)$ if and only if $X(e)\in \bighplus_{j\in B\cap\supp{X}}X(j)R_j(e)$. It then follows from Lemma~\ref{lem:orbitreps} that $X\in\CC(\MM)^\perp$ if and only if, for every basis $B$, $X\in \bighplus_{j\in B\cap\supp{X}}X(j)R_j$, and this is equivalent to $X$ being a linear combination of $\{R_j:j\in B\}$ by Lemma~\ref{lem:coeffs}.

Let $X\in F^E$ and $Y=FC(e,B)\in\CC(\MM)$ for some basis $B$ and some $e$. Consider $\{R_j: j\in B\}\subseteq \CC^*(\MM)$, the reduced row-echelon form with respect to $B$. Since $R_j\perp Y$ and $\supp{R_j}\cap\supp{Y}\subseteq \{j,e\}$, we have that if $j\in\supp{Y}$ then 
$$R_j(e)Y(e)^c+ R_j(j)Y(j)^c=R_j(e)Y(e)^c+ Y(j)^c\in N_G$$ and so  $Y(j)^c=-R_j(e)Y(e)^c$.
Further, if $j\in B-\supp{Y}$ then since $Y(e)\neq 0$ we have that $R_j(e)=0$.

Thus 
\begin{align*}
X\htimes Y&=\sum_{j\in\supp{X}\cap\supp{Y}} X(j)Y(j)^{c}\\[12pt]
&=X(e)Y(e)^{c}+\sum_{j\in B\cap\supp{X}\cap\supp{Y}}-X(j)R_j(e)Y(e)^{c}\\[12pt]
&=(X(e)+\sum_{j\in B\cap\supp{X}\cap\supp{Y}}-X(j)R_j(e))Y(e)^{c}
\end{align*}
and so $X\perp Y$ if and only if $X(e)+\sum_{j\in B\cap\supp{X}\cap\supp{Y}}-X(j)R_j(e)\in N_G$. This is true if and only if $X(e)\in \bighplus_{j\in B\cap\supp{X}\cap\supp{Y}}X(j)R_j(e)$. Since $R_j(e)=0$ for every $j\in B-\supp{Y}$, we have that $X\perp Y$ if and only if $X(e)\in \bighplus_{j\in B\cap\supp{X}}X(j)R_j(e)$. 

\end{proof}

\begin{defn} Let $\MM$ be a strong $F$-matroid. 

The set $\VV^*(\MM)$ of {\bf $F$-covectors} of $\MM$ is $\CC(\MM)^\perp$. Equivalently, $\VV^*(\MM)$ is the set of all elements of $F^E$ that are consistent with $\CC^*(\MM)$.

The set $\VV(\MM)$ of {\bf $F$-vectors} of $\MM$ is $(\CC^*(\MM))^\perp$. Equivalently, $\VV(\MM)$ is the set of all elements of $F^E$ that are consistent with $\CC(\MM)$.
\end{defn}

\begin{thm}\label{thm:cryptom} If $\MM$ is a strong $F$-matroid then $\VV(\MM)$ and $\VV^*(\MM)$ are $F$-vector sets (in the sense of Definition~\ref{def:vectorax}). 

Further, every $F$-vector set is $\VV^*(\MM)$ for some strong $F$-matroid $\MM$, and 
$$\CC^*(\MM)=\minsupp(\VV^*(\MM)-\{\0\})$$  
$$\CC(\MM)=\minsupp((\VV^*(\MM))^\perp-\{\0\}).$$ 
\end{thm}

The characterization of strong $F$-matroids by $F$-covectors solidifies the assertion from the introduction that ``the concept of matroids over a tract generalizes the concept of linear subspaces of a vector space". Previously this assertion rested on Grassmann-Pl\"ucker functions (\cite{BB17}), but the discussion at the beginning of this section justifies the following. 

\begin{prop} Let $K$ be a field and $V$ a linear subspace of $K^E$. Then there is a strong $K$-matroid $\MM$ such that
\begin{enumerate}
\item $V=\VV^*(\MM)$
\item $V^\perp=\VV(\MM)$, and
\item The projective coordinates of the Pl\"ucker embedding for $V$ constitute a Grassmann-Pl\"ucker function for $\MM$.
\end{enumerate}
Further, every $K$-matroid arises in this way.
\end{prop}

Thus when $K$ is a field the Grassmann-Pl\"ucker function $\phr_\MM$ of a $K$-matroid $\MM$ can be viewed as an element of $\PP(\bigwedge^r K^{E\choose r})$ of the form $[v_1\wedge\cdots\wedge v_r]$, where $\{v_1, \ldots, v_r\}$ is a basis for the vector space $\VV^*(\MM)$. The algebraic structure of the exterior algebra $\bigwedge^r K^{E\choose r}$ gives a lovely relationship between $\phr_\MM$ and the vector space $\VV^*(\MM)$:
$$X\in\VV^*(\MM)\Leftrightarrow X\wedge \phr_\MM=0.$$
This begs for generalization to other tracts, and such a generalization has been accomplished for tracts arising from an idempotent semifield by Jeffrey and Noah Giansiracusa~\cite{gian}. The tracts $\K$ and $\T\triangle$ are examples.  Their paper is not written in the language of matroids over tracts, but the characterization of $F$-matroids via Grassmann-Pl\"ucker functions in~\cite{BB17} allows one to characterize Proposition 4.2.1 in \cite{gian} as: for a tract $F$ arising from an idempotent semifield and a rank $r$  $F$-matroid $\MM$, there is an associated graded  ``tropical Grassmann algebra" $\bigwedge\MM$ such that $\phr_\MM\in\PP(\bigwedge^r\MM)$ and, for all $X\in F^E$,
$$X\in\CC(\MM)^\perp\Leftrightarrow X\wedge\phr_\MM=0.$$

\begin{remark} \label{rem:simpledef} The role of reduced row-echelon forms in  Definition~\ref{def:vectorax}  may seem like overkill. One might hope for a definition that says, for every maximal linearly independent subset $S$ of $\WW$,  every element of $\WW$  has a unique expression as a linear combination of $S$. Sadly, this is not equivalent to Definition~\ref{def:vectorax}, as will be illustrated by an example in Section~\ref{sec:lincomb}.
\end{remark}

\begin{remark} Baker and Bowler (\cite{BB17}, Section 3) define the vectors and covectors of an $F$-matroid $\MM$ to be $\CC^*(\MM)^\perp$ resp.\ $\CC(\MM)^\perp$, without giving an independent axiomatization of vectors and covectors.
\end{remark}

\section{Proof of Theorem~\ref{thm:cryptom} }
 We begin by proving that if $\WW$ is an $F$-vector set then $\minsupp(\WW-\{\0\})$ satisfies the $F$-circuit axioms.
 
 \begin{lemma}\label{lem:symmetry} [Symmetry for $F$-vector sets] If $\WW\subseteq F^E$ satisfies the tract vector axioms, $X\in\WW$, and $\alpha\in F$ then $\alpha X\in\WW$.
 \end{lemma}
 
 \begin{proof} If $B$ is a support basis and $X\in\bighplus_{j\in B} \beta_j R_j$ then  $\alpha X\in\bighplus_{j\in B} \alpha\beta_j R_j$.
 \end{proof}
 
\begin{lemma} \label{lem:existbasis} Let $\WW$ be an $F$-vector set.

1. If $X\in\minsupp(\WW-\{\0\})$ and $b\in\supp{X}$ then there is a support basis $B$ containing $b$ such that $X$ is a multiple of $R_b$.

2. If $\{X_1, \ldots, X_k\}\subseteq\minsupp(\WW-\{\0\})$ is a modular family and, for each $j\in[k]$, $e_j\in \supp{X_j}-\bigcup_{l\neq j}\supp{X_l}$ then there is a support basis $B\supseteq\{e_1, \ldots, e_k\}$ such that each $X_j$ is a multiple of $R_{e_j}$.

3. If $\{X_1, \ldots, X_k,X\}\subseteq\minsupp(\WW-\{\0\})$ is a modular family, for each $j\in[k]$, $e_j\in \supp{X_j}-\bigcup_{l\neq j}\supp{X_l}$, and $f\in\supp{X}-\bigcup_{j=1}^k\supp{X_j}$ then there is a support basis $B\supseteq\{e_1, \ldots, e_k,f\}$ such that each $X_j$ is a multiple of $R_{e_j}$ and $X$ is a linear combination of $\{X_1, \ldots, X_k, R_f\}$.
\end{lemma}

\begin{proof} 1. There is no $Z\in\WW$ such that $X^0\cup\{b\}\subseteq Z^0$, since otherwise $X$ does not have minimal support. Thus $X^0\cup\{b\}$ contains a support basis $B$ that has $b$ as an element, and uniqueness of reduced row echelon forms (Proposition~\ref{prop:urref}) implies $R_b=X(b)^{-1}X$.

2. We first claim that $(\bigcap_{j=1}^k X_j^0)\cup\{e_1, \ldots, e_k\}$ has nontrivial intersection with the support of each $Z\in\minsupp(\WW-\{\0\})$. If $\supp{Z}\cap\bigcap_{j=1}^k X_j^0=\emptyset$ then $\supp{Z}\subseteq \bigcup_{j=1}^k\supp{X_j}$. Let $l$ be minimal such that $\supp{Z}\subseteq \bigcup_{j=1}^l\supp{X_j}$. Then either $l=1$, in which case $\supp{Z}=\supp{X_1}$, and hence $e_1\in Z$, or 
$$\bigcup_{j=1}^{l-1}\supp{X_j}\subset \bigcup_{j=1}^{l-1}\supp{X_j}\cup\supp{Z}\subseteq\bigcup_{j=1}^{l}\supp{X_j}.$$ By modularity the second inclusion must in fact be an equality, and so $e_l\in\supp{Z}$.

Thus  $(\bigcap_{j=1}^k X_j^0)\cup\{e_1, \ldots, e_k\}$ contains a support basis $B$, and certainly $\{e_1, \ldots, e_k\}\subseteq B$. Uniqueness of reduced row echelon forms then implies that each $X_j$ is a multiple of $R_{e_j}$.

3. By the same argument as in (2) we see that $(X^0\cap \bigcap_{j=1}^k X_j^0)\cup\{e_1, \ldots, e_k,f\}$ contains a support basis $B$ so that $\{e_1, \ldots, e_k\}\subseteq B$. Then 
$$X\in\bighplus_{b\in B} X(b)R_b=\bighplus_{b\in B\cap\{e_1, \ldots, e_k, f\}} X(b)R_b.$$
By modularity $\supp{X}\not\subseteq\bigcup_{j=1}^k\supp{X_j}$, and so $X\not\in\bighplus_{b\in\{e_1,\ldots, e_k\}} X(b)R_b$. Thus $f\in B$ and we have our desired result.

\end{proof}

\begin{lemma}\label{lem:easycircaxs} If  $\WW$ is an $F$-vector set then $\minsupp(\WW-\{\0\})$ satisfies Symmetry and Incomparability.
\end{lemma}

\begin{proof} Symmetry follows from Lemma~\ref{lem:symmetry}.

To see Incomparability: let $X,Y\in\minsupp(\WW-\{\0\})$ such that $\supp{X}\subseteq\supp{Y}$, and let $b\in\supp{X}$. By definition $\supp{X}=\supp{Y}$. Consider a support basis $B$ contained in $X^0\cup\{b\}$. This exists by Lemma~\ref{lem:existbasis}, and also by Lemma~\ref{lem:existbasis} $X=X(b) R_b$ and $Y=Y(b)R_b$. Thus $X=X(b)Y(b)^{-1}Y$.

\end{proof}

\begin{lemma} \label{cor:modelim} If $\WW$ is an $F$-vector set then $\minsupp(\WW-\{\0\})$ satisfies the strong $F$-circuit axioms.
\end{lemma}

\begin{proof} Lemma~\ref{lem:easycircaxs} proved everything except Strong Modular Elimination. Let $\{X_1, \ldots, X_k, X\} \subseteq \minsupp(\WW-\{\0\})$ and $\{e_1, \ldots, e_k\}$ be as in the hypothesis of Strong Modular Elimination. Let $f\in \supp{X}\backslash\cup_{j=1}^k\supp{X_j}$. By Lemma~\ref{lem:existbasis}.3 we get a support basis $B\supseteq\{e_1, \ldots, e_k, f\}$ such that $$X\in X(f) R_f\hplus\bighplus_{j=1}^k X(e_j) R_{e_j}.$$ But $X_j=X_j(e_j)R_{e_j}=-X(e_j)R_{e_j}$, and so
\begin{align*}
X\in&X(f)R_f\hplus\bighplus_{j=1}^k -X_j\\
X(f)R_f\in&X\hplus\bighplus_{j=1}^k X_j
\end{align*}
and so $X(f)R_f$ is our desired elimination $Z$.

\end{proof}

\begin{lemma} \label{lem:rrefscoincide} If $\MM$ is an $F$-matroid then the reduced row-echelon forms for $\VV^*(\MM)$ are  exactly the reduced row echelon forms for $\CC^*(\MM)$.
\end{lemma}

\begin{proof} Recall (Lemma~\ref{lem:supportbasis1}.1) that the support bases for $\VV^*(\MM)$ are exactly the support bases of $\CC^*(\MM)$. Thus it suffices to show that every element of every reduced row echelon form for $\VV^*(\MM)$ is in $\CC^*(\MM)$. Let $\{R_j:j\in B\}$ be a reduced row-echelon form for $\CC^*(\MM)$ with respect to a basis $B$, and let $X$ be an element of some reduced row-echelon form for $\VV^*(\MM)$ with respect to $B$. Then $\supp{X}\cap B=\{j_0\}$ for some $j_0$, and $X\in\bighplus_{j\in B} X(j)R_j=X(j_0) R_{j_0}$. Since $X$ is in a reduced row-echelon form with respect to $B$, we have that $X(j_0)=1$, and so $X=R_{j_0}$.
\end{proof}

We now complete the proof of Theorem~\ref{thm:cryptom}.

\begin{proof} Let $\MM$ be a strong $F$-matroid.  Then $\VV^*(\MM)=\CC(\MM)^\perp$, and so $\minsupp(\VV^*(\MM)-\{\0\})=\minsupp(\CC(\MM)^\perp-\{\0\})$, which is $\CC^*(\MM)$ by definition of $\CC^*$.  
 $\VV^*(\MM)$ has a reduced row-echelon form with respect to every basis,  by Lemmas~\ref{lem:circrref} and~\ref{lem:rrefscoincide},  and these are exactly the reduced row-echelon forms for $\CC^*(\MM)$. Thus $\VV^*(\MM)$, by  definition, satisfies  the second form of the tract Vector Axiom. By duality, $\VV(\MM)=\VV^*(\MM^*)$ does as well.
 
 Lemma~\ref{cor:modelim} showed that if $\WW\subseteq F^E$ is an $F$-vector set then $\minsupp(\WW-\{\0\})=\CC^*(\MM)$ for some $\MM$.  Thus, by Lemma~\ref{lem:supportbasis1}, the set of support bases of $\WW$ is exactly the set of bases of $\MM$, and so the definition of $\WW$ tells us that $\WW$ is exactly the set of elements of $F^E$ that are consistent with $\CC^*(\MM)$.  Thus by definition $\WW=\VV^*(\MM)$.

Finally, to see that $\CC(\MM)=\minsupp(\WW^\perp-\{\0\})$: since $\WW=\CC(\MM)^\perp$, we have $\CC(\MM)\subseteq\WW^\perp-\{\0\}$. 

To see that $\CC(\MM)\subseteq\minsupp(\WW^\perp-\{\0\})$, let $X\in\CC(\MM)$ and $Y\in\WW^\perp-\{\0\}$ such that $\supp{Y}\subseteq\supp{X}$. Let $e\in\supp{Y}$ and $f\in\supp{X}-\{e\}$. Then Lemma~\ref{lem:fundcirc} says that $X$ is a multiple of $FC(f,B)$ for some basis $B$. This $B$ contains $e$, and so the reduced row-echelon form $\{R_j:j\in B\}\subseteq\CC^*(\MM)$ contains an element $R_e$. Note $\supp{R_e}\cap\supp{X}=\{e,f\}$. Thus 
 $e\in\supp{Y}\cap\supp{R_e}\subseteq\{e,f\}$, and so orthogonality of $Y$ and $R_e$ implies that $Y(e)=Y(e)R_e(e)^c=-Y(f)R_e(f)^c$. Likewise $X(e)=X(e)R_e(e)^c=-X(f)R_e(f)^c$, and so
$$-\frac{1}{R_e(f)^c}=\frac{X(f)}{X(e)}=\frac{Y(f)}{Y(e)}$$
$$Y(f)=\frac{Y(e)}{X(e)}X(f)$$
and so $Y=\frac{Y(e)}{X(e)}X$. In particular, $\supp{Y}=\supp{X}$.

To see that $\minsupp(\WW^\perp-\{\0\})\subseteq\CC(\MM)$, we recall that the circuits of the underlying matroid of $\MM$ are exactly the minimal subsets $S$ of $E$  such that, for each cocircuit $T$, $|S\cap T|\neq 1$. Also recall (Proposition~\ref{prop:morph}) that the circuits resp. cocircuits of the underlying matroid of $\MM$ are exactly the supports of the $F$-circuits resp. $F$-cocircuits of $\MM$. If $X\in \WW^\perp-\{\0\}$ then certainly $|\supp{X}\cap\supp{Y}|\neq 1$ for each $Y\in\CC^*(\MM)$. Thus $\supp{X}$ contains the support of some element of $\CC(\MM)$. In particular, if $X\in\minsupp(\WW^\perp-\{\0\})$ then $X\in\CC(\MM)$.
 \end{proof}

\section{Basic properties}

\subsection{Duality}\label{sec:duality}
 
 Our results so far show that $\VV^*(\MM)=\CC(\MM)^\perp\supseteq\VV(\MM)^\perp$. If $F$ is a field then the  inclusion $\VV^*(\MM)\supseteq\VV(\MM)^\perp$ is an equality: this is just ordinary orthogonality of vector spaces. This is also the case if $F=\S$ (the oriented matroid case) -- see Example~\ref{ex:OMcomp}.
 
{\em However, in general $\VV(\MM)^\perp$ might be a proper subset of $\VV^*(\MM)$.}  This will be shown by an example with $F=\P$ in Section~\ref{sec:dualitycounterex}.

\begin{defn} (\cite{BB17}, see also \cite{DWperfect} ) A tract $F$ is {\bf perfect} if $\VV(\MM)^\perp=\VV^*(\MM)$ for every $F$-matroid $\MM$.
\end{defn}

In~\cite{BB17} Baker and Bowler give various sufficient conditions for a tract to be perfect.

\subsection{Deletion and contraction}\label{sec:delcontr}

Deletion and contraction  of $F$-matroids are defined in a way consistent with matroid and oriented matroid definitions:

\begin{thm}[\cite{BB17}, 2.29] \label{bakerdelcontr} 
Let $\MM$ be an $F$-matroid on $E$, and let $A\subseteq E$. Then $\minsupp(\{X\backslash A: X\in\CC(\MM)\}-\{\0\})$ is the set of $F$-circuits of an $F$-matroid $\MM/A$, called the {\bf contraction} of $\MM$ with respect to $A$, and $\{X\backslash A: X\in\CC(\MM), \supp{X}\cap A=\emptyset\}$ is the set of $F$-circuits of an $F$-matroid $\MM\backslash A$, called the {\bf deletion} of $\MM$ with respect to $A$. Moreover, $(M/A)^*=M\backslash A$ and $(M\backslash A)^*=M/A$.
\end{thm}

It turns out that the covector interpretations of deletion and contraction are  more problematic. In  contrast to the situation with matroids and oriented matroids,  {\em it is not always true that $\VV^*(\MM\backslash e)=\{X\backslash e: X\in\VV^*(\MM)\}$, nor that $\VV^*(M/e)=\{X\backslash e: X\in\VV^*(\MM), X(e)=0\}$.} This is illustrated by examples  in Sections~\ref{sec:nodel} and~\ref{sec:nocontr}. The best results we have are the following three propositions. (Recall the notation $Xe^0$ and $X\backslash e$ from the beginning of Section~\ref{sec:background}.)

\begin{prop} \label{prop:contraction} $\{X\backslash e: X\in\VV^*(\MM), X(e)=0\}\subseteq\VV^*(\MM/e)$ and
$\{X\backslash e: X\in\VV(\MM), X(e)=0\}\subseteq\VV(\MM\backslash e)$.
\end{prop}

\begin{proof} %By Theorem~\ref{thm:cryptom}, $\VV^*(\MM)=\CC(\MM)^\perp$. Thus i
If $X\in\VV^*(\MM)$ then 
$ X\perp Y$ for all $Y\in\CC(\MM)$, and so if in addition $X(e)=0$ then $(X\backslash e)\perp (Y\backslash e)$ for all such $Y$. Since $\CC(\MM/e)\subseteq \{Y\backslash e:Y\in\CC(\MM)\}$, we have $X\backslash e\in\VV^*(\MM/e)$.
The second result follows because $\VV(\MM^*)=\VV^*(\MM)$ and $\MM^*/ e=(\MM\backslash e)^*$.
\end{proof}

\begin{prop}\label{prop:deletion}  $\{X\backslash e: X\in\VV^*(\MM)\}\subseteq \VV^*(\MM\backslash e)$ and $\{X\backslash e: X\in\VV(\MM)\}\subseteq \VV(\MM/ e)$
\end{prop}

\begin{proof} By Theorem~\ref{bakerdelcontr}, $\CC(\MM\backslash e)=\{X\backslash e: X\in\CC(\MM), X(e)=0\}$. Thus if $Y\in\CC(\MM)^\perp$ then $Y\backslash e\in\CC(\MM\backslash e)^\perp$. 
The second result follows because $\VV(\MM^*)=\VV^*(\MM)$ and $\MM^*\backslash e=(\MM/e)^*$.
\end{proof}

The following result is due to Ting Su.
\begin{prop} If $F$ is perfect then 
\begin{enumerate}
\item $\VV^*(M/e)=\{X\backslash e: X\in\VV^*(\MM), X(e)=0\}$.
\item $\VV(M\backslash e)=\{X\backslash e: X\in\VV(\MM), X(e)=0\}$.
\item $\VV^*(\MM\backslash e)^\perp=\{X\backslash e: X\in\VV^*(\MM)\}^\perp$ 
\item $\VV(\MM/ e)^\perp=\{X\backslash e: X\in\VV(\MM)\}^\perp$ 

%\item $\VV^*(\MM\backslash e)=\{X\backslash e: X\in\VV^*(\MM)\}$ 
%\item $\VV(\MM/ e)=\{X\backslash e: X\in\VV(\MM)\}$ 
\end{enumerate}
\end{prop}

\begin{proof} By Proposition~\ref{prop:deletion} we have that $\{X\backslash e: X\in\VV(\MM)\}\subseteq \VV(\MM/ e)$, and so% $\VV(\MM/ e)^\perp\subseteq \{X\backslash e: X\in
\begin{align*}
\VV^*(\MM/e)&=\VV(\MM/ e)^\perp\\
&\subseteq \{X\backslash e: X\in\VV(\MM)\}^\perp\\
&=\{Y: Y\perp X\backslash e\mbox{ for all $X\in\VV(\MM)$}\}\\
 &=\{Y: Ye^0\perp X\mbox{ for all $X\in\VV(\MM)$}\}\\
 &=\{Y'\backslash e: Y'\in\VV(\MM)^\perp, Y'(e)=0\}\\
 &=\{Y'\backslash e: Y'\in\VV^*(\MM), Y'(e)=0\}\\
 &\subseteq \VV^*(\MM/e).
 \end{align*}
 and thus the two subset relations must be equalities. The second relation is Result (1), and the first relation is Result (3). Result (2) follows from (1)   and Result (4) follows from (3).

  \end{proof}

It would be interesting to find direct characterizations of $\VV^*(\MM\backslash e)$ and $\VV^*(\MM/e)$ in terms of $\VV^*(\MM)$ for arbitrary matroids over tracts, and to fully characterize the  tracts $F$ such that, for all $F$-matroids $\MM$ and elements $e$ of $\MM$, 
$\VV^*(\MM\backslash e)=\{X\backslash e: X\in\VV^*(\MM)\}$ and $\VV^*(M/e)=\{X\backslash e: X\in\VV^*(\MM), X(e)=0\}$.

\subsection{Morphisms of tracts}\label{sec:morph}

\begin{prop} \label{prop:pushforward} If $f:F\to F'$ is a morphism of tracts and $\MM$ is a strong $F$-matroid then $\VV(f_*(\MM))\supseteq\{\alpha f(X): X\in\VV(\MM), \alpha\in F'\}$.
\end{prop}

\begin{proof} By Proposition~\ref{prop:morph}, $\CC^*(f_*(\MM))=\{\alpha f(Y): Y\in\CC^*(\MM), \alpha\in F'-\{0\}\}.$ Also, if $X,Y\in F^E$ and $X\perp Y$ then $f(X)\perp f(Y)$. Thus 
\begin{align*}
(\CC^*(f_*(\MM)))^\perp&=\{ \alpha f(Y): Y\in\CC^*(\MM), \alpha\in F'-\{0\} \}^\perp\\
&\supseteq \{\alpha f(X):X\in\CC^*(\MM)^\perp,  \alpha\in F'\}\\
&=\{\alpha f(X):X\in\VV(\MM),  \alpha\in F'\}\qedhere
\end{align*}
\end{proof}

In general, equality will not hold. For instance, consider the inclusion $\iota:\R\to\C$ and an $\R$-matroid $\MM$ with $\VV(\MM)$ a two-dimensional subspace of $\R^n$ spanned by $\{v,w\}$. Then $\{\alpha\iota(X): X\in\VV(\MM), \alpha\in \C-\{0\}\}$ has topological dimension 3 in $\C^n$, but $\VV(\iota_*(\MM))$ is the complexification of $\VV(\MM)$, hence is a rank 2 linear subspace of $\C^n$, hence has topological dimension 4.

A bit more surprisingly, equality need not hold even when the morphism of tracts is surjective:
\begin{example} \label{ex:morph}
Consider the tract morphism $\ph:\C\to \P$ and the $\C$-matroids $\MM_1$ and $\MM_2$ with 
$$\VV^*(\MM_1)=\begin{pmatrix}
1&1+i&1&0\\
1+i&3i&0&1
\end{pmatrix}$$
and 
$$\VV^*(\MM_2)=\begin{pmatrix}
1&1+i&1&0\\
1+i&4i&0&1
\end{pmatrix}.$$
As shown in~\cite{AD},
$$\{\ph(X): X\in\CC^*(\MM_1)\}=\{\ph(X): X\in\CC^*(\MM_2)\}$$
but
$$\{\ph(X): X\in\VV^*(\MM_1)\}\not\subseteq\{\ph(X): X\in\VV^*(\MM_2)\}.$$
Thus $\ph_*(\MM_1)=\ph_*(\MM_2)$ and 
\begin{align*}
\VV^*(\ph_*(\MM_1))&\supseteq\{\ph(X): X\in\VV^*(\MM_1)\}\cup\{\ph(X): X\in\VV^*(\MM_2)\}\mbox{ by Proposition~\ref{prop:pushforward}}\\
&\supsetneq\{\ph(X): X\in\VV^*(\MM_2)\}\\
&=\{\alpha\ph(X): X\in\VV^*(\MM_2),\alpha\in\P-\{0\}\}.
\end{align*}
\end{example}
\subsection{Loops and coloops}

\begin{defn} A {\bf loop}\ resp.\  {\bf coloop}\  of an $F$-matroid is a loop resp.\  coloop of the underlying matroid.
\end{defn}

\begin{lemma} \label{lem:coloop} 1. $e$ is a loop of $\MM$ if and only if $X(e)=0$ for every $X\in\VV^*(\MM)$.

2. $e$ is a coloop of $\MM$ if and only if $\VV^*(\MM)=\{Xe^\alpha: X\in\VV^*(\MM\backslash e), \alpha\in F\}$.
\end{lemma}

\begin{proof} 1. $e$ is a loop of $\MM$ if and only if $\{e\}$ is a circuit of the underlying matroid, hence if and only if there is a $Y\in\CC(\MM)$ with $\supp{Y}=\{e\}$. But then, for every $X\in\VV^*(\MM)$, $X\perp Y$  implies that $X(e)=0$.

2. Recall $\CC(\MM\backslash e)=\{Y\backslash e: Y\in\CC(\MM)\}$. Also, $e$ is a coloop of $\MM$ if and only if $e$ is not in any circuit of the underlying matroid. Thus
\begin{align*}
\mbox{$e$ is a coloop of $\MM$}&\Leftrightarrow Y(e)=0\ \forall Y\in\CC(\MM)\\
&\Leftrightarrow\CC(\MM)^\perp=\{Xe^\alpha: X\in\CC(\MM\backslash e)^\perp, \alpha\in F\}\\
&\Leftrightarrow\VV^*(\MM)=\{Xe^\alpha: X\in\VV^*(\MM\backslash e), \alpha\in F\}.\qedhere
\end{align*}
\end{proof}

\section{Examples}

\subsection{Rank 1 $F$-matroids}

Clearly rank 1 weak $F$-matroids are also strong $F$-matroids. Not surprisingly, their $F$-vectors are easy to describe.
\begin{prop}\label{prop:rank1} Let $F=G\cup\{0\}$ be a tract and $E$ a finite set.
\begin{enumerate}
\item Every nonzero $\phr\in F^E$ is a Grassmann-Pl\"ucker function for a rank 1 $F$-matroid.
\item If $\MM$ is a rank 1 $F$-matroid with Grassmann-Pl\"ucker function $\phr$ then 
\begin{enumerate}
\item $\CC^*(\MM)=\VV^*(\MM)-\{\0\}=G \phr$.
\item $\CC(\MM)=\{\alpha X_{e,f}:\phr(e)\phr(f)\neq 0, \alpha\in G\}\cup\{\alpha Y_e:\phr(e)=0, \alpha\in G\}$, where 
\begin{align*}
X_{e,f}(e)^c&=\phr(f)^{-1} \\
X_{e,f}(f)^c&=-\phr(e)^{-1} \\
X(g)&=0\qquad\mbox{if $g\not\in\{e,f\}$, }
\end{align*}
 and 
 \begin{align*}
 Y_e(e)&=1\\
 Y_e(g)&=0\qquad\mbox{ for all  $g\neq 0$.}
 \end{align*}
\item $\VV(\MM)=\{\phr\}^\perp$.
\end{enumerate}
\end{enumerate}
\end{prop}
This follows immediately from the definitions of the various objects and the crytomorphisms in~(\cite{BB17}).

\subsection{Matroids}

\begin{prop}\label{prop:matroidvect} Let $\MM$ be a $\K$-matroid. Then 
$$\VV^*(\MM)=\{X\in\K^E: \supp{X}\mbox{ is a union of cocircuits of $\MM$}\}.$$
\end{prop}

\begin{proof} If $\supp{X}=\bigcup_{j=1}^k X_j$ with each $X_j\in\CC^*(\MM)$,   then consider $Y\in\CC(\MM)$. Note that two elements $A,B$ of $\K^E$ are orthogonal if and only if $|\supp{A} \cap\supp{B}|\neq 1$. Since $\supp{X}\cap\supp{Y}=\bigcup_{j=1}^k(\supp{X_j}\cap\supp{Y})$ and $|\supp{X_j}\cap\supp{Y}|\neq 1$ for every $j$, we have that $X\perp Y$. Hence $X\in\CC(\MM)^\perp=\VV^*(\MM)$.

To see the converse, we induct on the number of nonloops in $X^0$ for $X\in\VV^*(\MM)$. If $X^0$ contains only loops, then $\supp{X}$ is the union of all cocircuits of $\MM$. Otherwise, let $e\in X^0$ be a nonloop. Then 
$X\backslash e\in\VV^*(\MM/e)$
by Proposition~\ref{prop:contraction}, and by our induction hypothesis 
$\supp{X\backslash e}=\supp{X_1}\cap\cdots\cap\supp{X_k}$ for some $\K$-cocircuits $X_1,\ldots, X_k$ of $\MM\backslash e$. But then $X_1e^0,\ldots, X_ke^0\in\CC^*(\MM)$ by Theorem 2.29 of~\cite{BB17},
and $\supp{X}$ is the union of the supports of these.
  \end{proof}

Another explanation for this equality can be given by the Inflation Property (Section~\ref{sec:absorp}).

\subsection{Oriented matroids}

\begin{prop}\label{cor:OM} $\WW\subseteq \S^E$ is the set of  $\S$-vectors of an $\S$-matroid with $\S$-circuit set $\CC$ if and only $\WW$ is the set of signed vectors of an oriented matroid with  signed circuit set $\CC$.
\end{prop}

\begin{proof} For $F=\S$ the $F$-circuit axioms coincide with the usual signed circuit axioms (Corollary 1 in~\cite{Del11} and Theorem 3.6.1 in~\cite{BLSWZ}). Further, $\WW$ is the set of signed vectors of an oriented matroid if and only if $\WW=(\CC^*)^\perp$ for some signed circuit set $\CC^*$ (Proposition 3.7.12 in~\cite{BLSWZ}). In either the usual oriented matroid context of the $\S$-matroid context, the set of signed circuits/$\S$-circuits corresponding to $\WW$ is $\minsupp(\WW-\{\0\})$.
\end{proof}

\subsection{Phased matroids}\label{sec:phased}

Phased matroids (i.e., $\P$-matroids) seem to be where conjectures on this subject go to die. This section will give examples to show that various properties satisfied by $F$-covector sets when $F$ is a field or $F=\S$ fail to hold when $F=\P$.

\subsubsection{Topological closure}

If $F$ is a topological field and $S\subset F^E$ then $S^\perp$ is closed. In particular, each $F$-covector set $\VV^*(\MM)$ is closed.
This does not generalize in any nice way to tracts.

Associating topologies to tracts is a tricky business (cf. Section 8 of~\cite{Viro1}, \cite{ADavis}, \cite{jun}). Any topology on a tract $F$ induces a topology on $F^E$ and hence on each $F$-covector set. As the following shows, in order to choose a topology on $F$ that makes all $F$-covector sets closed, we may have to sacrifice obvious desirable properties.

\begin{prop} 1. If $\S$ is given a topology such that $\sign:\R\to\S$ is continuous then there is a rank 2 $\S$-matroid $\MM_1$ on 3 elements such that $\VV^*(\MM_1)$ is not closed in $\S^3$.

\noindent 2.  If $\P$ is given a topology such that $\ph:\C\to\P$ is continuous then there is a rank 2 $\P$-matroid $\MM_2$ on 3 elements such that $\VV^*(\MM_2)-\{(0,0,0)\}$ is not closed in $(\P-\{0\})^3$.
\end{prop}

\begin{proof} Let $${\mathbf A}=\begin{pmatrix}
1&0&1\\
0&1&1
\end{pmatrix}.$$

Assume $\S$ is given a topology such that $\sign:\R\to\S$ is continuous. Thus the only open neighborhood of 0 in $\S$ is $\S$. Let $\MM_1'$ be the $\R$-matroid with $\VV^*(\MM_1')=\row(A)$, and let $\MM_1=\sign_*(\MM_1')$.  The triple $(+,+,0)$ is in the complement of $\VV^*(\MM_1)$, but any open neighborhood of $(+,+,0)$ in $\S^3$ contains $(+,+,+)\in\VV^*(\MM_1)$. Thus the complement of $\VV^*(\MM_1)$ is not open.

Assume $\P$ is given a topology such that $\ph:\C\to\P$ is continuous. Thus if $x\in\P-\{0\}$ then any open neighborhood of $x$ in $\P$ contains a set that is an open neighborhood of $x$ with respect to the usual topology on the unit circle in $\C$. Let $\MM_2'$ be the $\C$-matroid with $\VV^*(\MM_2')=\row(A)$, and let $\MM_2=\ph_*(\MM_2')$.   It is a straightforward exercise to check that $\VV^*(\MM_2)=\{(\alpha, \beta, \gamma)\in\P^3:\gamma\in \alpha\hplus \beta\}$.
Thus the triple $(1,i,1)$ is in the complement of $\VV^*(\MM_2)$, but any open neighborhood of $(1,i,1)$ in $(\P-\{0\})^3$ contains triples $(1,i,\alpha)\in\VV^*(\MM_2)$. Thus the complement of $\VV^*(\MM_2)-\{(0,0,0)\}$ in $(\P-\{0\})^3$ is not open.
\end{proof}

%
%
%Each of the tracts of Example~\ref{ex:tracts} has a topology as subspaces of $\R$ or $\C$, and hence the  $F$-covector sets of matroids over each of  these $F$ have topologies as subspaces of $F^E$.
%  In contrast to the situation in topological vector spaces,  $\VV^*(\MM)=\CC(\MM)^\perp$ does {\em not} imply that $\VV^*(\MM)$ is topologically closed. For instance, consider $\P$ with topology as a subspace of $\C$. If $\MM$ is the $\C$-matroid with
%$$\VV^*(\MM)=\row\left(\begin{array}{ccc}
%1&0&i\\
%0&1&1
%\end{array}\right)$$
%and $\MM'=\ph_*(\MM)$ 
%then it is a straightforward exercise to check that $\VV^*(\MM')=\{(\alpha, \beta, \gamma)\in\P^3:\gamma\in (i\htimes\alpha)\hplus \beta\}$. Note that $i\hplus 1=\{\exp(i\theta):0<\theta<\pi/2\}$. Thus
%$(1,1,i)$ is in the topological closure of $\VV^*(\MM')$ but is not in $\VV^*(\MM')$. 
 
\subsubsection{ Linear independence  counterexample}\label{sec:lincomb}

Let $\MM_1$ and $\MM_2$ be as in Example~\ref{ex:morph}, and let $\MM=\ph_*(\MM_1)=\ph_*(\MM_2)$. This $\MM$  will shoot down the hope for a simpler characterization of $F$-vector sets expressed in Remark~\ref{rem:simpledef}.

\begin{defn}[cf. \cite{BB17}] $\{X_1, \ldots, X_k\}\subseteq F^E$ is {\bf linearly dependent}  if there exist elements  $c_1, \ldots, c_k$ of $F$,  not all $0$, such that $\0\in\bighplus_{j=1}^k c_j X_j$.

A subset of $F^E$ is {\bf linearly independent} if it is not linearly dependent.
\end{defn}

Linear independence behaves badly. A $\P$-vector set may have maximal linearly independent subsets of different sizes.
%Not all maximal linearly independent subsets of a $\P$-vector set need necessarily have the same size. 
In Example~\ref{ex:morph},  $(2+i, 1+4i, 1,1)\in\VV^*(\MM_1)$ and 
$(2+i, 1+5i, 1,1)\in\VV^*(\MM_2)$, and so $X_1:=\ph(2+i, 1+4i, 1,1)$ and $X_2:=\ph(2+i, 1+5i, 1,1)$ are elements of $\VV^*(\MM)$. Each reduced row-echelon form for $\VV^*(\MM)$ is a maximal linearly independent set of size 2, but also $\VV^*(\MM)$ contains  $S:=\{X_1, X_2, \ph(1, 1+i, 1,0)\}$, which is  linearly independent.

Further, the element $\ph(1+i, 3i,0,1)$ of $\VV^*(\MM)$ can be expressed as a linear combination of $S$ both as
$$\ph(1+i, 3i,0,1)\in (-1)X_1\hplus 0X_2\hplus \ph(1, 1+i, 1,0)$$
and as 
$$\ph(1+i, 3i,0,1)\in 0X_1\hplus (-1)X_2\hplus \ph(1, 1+i, 1,0)$$
thus killing the  simpler characterization of $F$-vector sets hoped for in Remark~\ref{rem:simpledef}.

\subsubsection{Phase diagrams}

This section introduces a visualization tool that will be helpful in later counterexamples.

We will depict $X\in (S^1\cup\{0\})^n$ by labelled points on a picture of $S^1$. (If $X(f)=0$ then the label $f$ is not used.) We call this the {\bf phase diagram} for $X$. Thus, for instance, the leftmost circle in Figure~\ref{fig:vperpex} depicts $(1,0,\exp(i\pi/2),\exp(-i\pi/2))$. It is easy to see from the phase diagram, for instance, whether $X\in \P^n$ or $X\in\TP^n$ is orthogonal to $(1,1,\ldots, 1)$:
\begin{enumerate}
\item  In $\P^n$, $X\neq\0$ is orthogonal to $(1,1,\ldots, 1)$ if and only either the phase diagram has only two points (possibly with multiple labels), which are antipodal to each other, or the points in the phase diagram for $X$ do not all lie in a common closed half-circle.
\item In $\TP^n$, $X\neq\0$ is orthogonal to $(1,1,\ldots, 1)$ if and only if the points in the phase diagram for $X$ do not all lie in a common open half-circle.
\end{enumerate}

\subsubsection{Duality counterexample}\label{sec:dualitycounterex}

Now we give the example promised in Section~\ref{sec:duality}, showing that for some $\P$-matroids $\MM$, $\VV^*(\MM)\neq\VV(\MM)^\perp$.
 
Let\footnote{A Mathematica notebook verifying all of the linear algebra calculations of Section~\ref{sec:phased} is available on the arXiv.}
 $$V=\row\begin{pmatrix}
 1&0&1+i&1-i\\
 0&1&1-i&1+i
 \end{pmatrix}  
 =\row\begin{pmatrix}
  1-i&1+i&4&0\\
 1+i&1-i&0&4
 \end{pmatrix}.    $$
Let $\MM_V$ denote the $\C$-matroid with $\VV^*(\MM_V)=V$, and let $\MM:=\ph_*(\MM_V)$. 

We can read off our four $S^1$-orbits of $\P$-cocircuits of  $\MM$ from these four row vectors: one representative from each orbit is shown in Figure~\ref{fig:vperpex}.

\begin{figure}
\begin{tikzpicture}
\draw (5,5) circle (1cm);
\draw [fill] (6,5) circle (.1cm); \node at (6.4,5) {$1$};
\draw [fill] (5.7,5.7) circle (.1cm); \node at (6.1,5.8) {$3$};
\draw [fill] (5.7,4.3) circle (.1cm); \node at (6,4.2) {$4$};

\draw (8,5) circle (1cm);
\draw [fill] (9,5) circle (.1cm); \node at (9.4,5) {$2$};
\draw [fill] (8.7,5.7) circle (.1cm); \node at (9.1,5.8) {$4$};
\draw [fill] (8.7,4.3) circle (.1cm); \node at (9,4.2) {$3$};

\draw (11,5) circle (1cm);
\draw [fill] (12,5) circle (.1cm); \node at (12.4,5) {$3$};
\draw [fill] (11.7,5.7) circle (.1cm); \node at (12.1,5.8) {$2$};
\draw [fill] (11.7,4.3) circle (.1cm); \node at (12,4.2) {$1$};

\draw (14,5) circle (1cm);
\draw [fill] (15,5) circle (.1cm); \node at (15.4,5) {$4$};
\draw [fill] (14.7,5.7) circle (.1cm); \node at (15.1,5.8) {$1$};
\draw [fill] (14.7,4.3) circle (.1cm); \node at (15,4.2) {$2$};

\end{tikzpicture}
\caption{$\P$-cocircuits of $\MM$ for Section \ref{sec:dualitycounterex}\label{fig:vperpex}}
\end{figure}
From the figure it's clear that any element of $(S^1)^4$ sufficiently close to $(1,1,-1,-1)$ will be orthogonal to each of these $\P$-cocircuits, hence will be in  $\VV(\MM)$. For a concrete example, $X:=(1,1,\exp({i(\pi+.01)}),\exp({i(\pi+.01)}))\in\VV(\MM)$. But also $(1,1,2,2)\in V=\VV^*(\MM_V)$, and so by Proposition~\ref{prop:pushforward} $\ph(1,1,2,2)=(1,1,1,1)\in\VV^*(\MM)$. But $X\not\perp (1,1,1,1)$.

\subsubsection{Deletion counterexample}\label{sec:nodel}

Here is the horrible example promised in Section~\ref{sec:delcontr}, showing that $\VV^*(\MM\backslash e)$ need not be $\{X\backslash e: X\in\VV^*(\MM)\}$.

\begin{example}\label{ex:el:bad} (A slight variation on this example arose in~\cite{AD}.) Let
\[ V:=\row
\begin{pmatrix}
0 & -1 & 0 & 0 & i& 1-i &1\\
 -1& 0  &-1 & 0 &-i& 3+i &2\\
 0 & -i & 0 & 2i&-i& -2i &-i\\
 0 & 0  &-i &i+1& 0& -2  &-1\\
\end{pmatrix}\]
and let $\MM_V$ be the $\C$-matroid with $\VV^*(\MM_V)=V$.
Let $\MM=\ph_*(\MM_V)$. We will find an element of $\VV^*(\MM\backslash 7)$ that is not in $\{X\backslash 7: X\in\VV^*(\MM)\}$.

By definition, $\CC(\MM_V\backslash 7)= \minsupp(\{X\backslash 7: X\in V^\perp, X(7)=0\}-\{\0\})$. We list a representative from each 
$(\C-\{0\})$-orbit below. 
$$\begin{array}{rccccccl}
\big(&-1+i,&-i,&1,&\frac{1}{2}-\frac{i}{2},&1,&0&\big)\\ [12pt]
\big(&2-i,&1+i,&1,&\frac{3}{2}+\frac{i}{2},&0,&1&\big)\\ [12pt]
\big(&5-5i,&1+3i,&-2-2i,&0,&-3-i,&1-i&\big)\\ [12pt]
\big(&3-2i,&1+2i,&0,&1+i,&-1,&1&\big)\\ [12pt]
\big(&3+4i,&0,&5,&\frac{7}{2}+\frac{i}{2},&3 -i,&2+i&\big)\\ [12pt]
\big(&0,&7-4i,&13,&\frac{25}{2}-\frac{5i}{2},&8+i,&5-i&\big)
\end{array}
$$

 Let $Z_0=(1,1,-1,1,1,-1)\in\P^6$. It is easy to check, by drawing phase diagrams,  that any element of $(S^1)^6$ sufficiently close to $Z_0$  is orthogonal to $\ph(X)$ for each representative $X$ above, hence is in $\VV^*(\MM\backslash 7)$. In particular, if we let $\alpha=\exp(i\epsilon)$ with $\epsilon>0$ small, then  $Z_1=(\alpha,1,-1,1,\alpha,-1)$  is in  $\VV^*(\MM\backslash 7)$.

On the other hand, notice that 
$Y_1:=(1,1,1,1,0,0,1)$ and $Y_2:=(0,0,1,1,1,1,-1)$ are both in $\minsupp(V^\perp-\{\0\}))= \CC(\MM_V)$. Since $Y_1=\ph(Y_1)$ and $Y_2=\ph(Y_2)$, we have $Y_1, Y_2\in\CC(\MM)$. Consider a  $Z\in\P^7$ such that $Z\backslash 7=Z_1$. Then 
$$Z\htimes Y_1=\alpha+1+-1+1+ Z(7)$$
and $$Z\htimes Y_2=-1+1+ \alpha-1 -Z(7).$$
Thus
$Z$ is orthogonal to $Y_1$ if and only if  $Z(7)$ in the lower open half-circle of $S^1$, while $Z$ is orthogonal to $Y_2$ if and only if  $Z(7)$ in the upper open half-circle of $S^1$. Thus $Z_1\in\VV^*(\MM\backslash 7)- \{Z\backslash 7: Z\in\VV^*(\MM)\}$.
\end{example}

\subsubsection{Contraction counterexample}\label{sec:nocontr}
 This section gives an example of a rank 3 $\P$-matroid $\MM'$ on elements $[6]$ such that $\VV^*(\MM'/6)\supsetneq \{X\backslash 6: X\in\VV^*(\MM'), X(6)=0\}$.
 
Let
\[ W:=\row
\begin{pmatrix}
3&0&0&1&1&-3\\
0&3&0&1&1&3+3i\\
0&0&3&1&1&3-3i
\end{pmatrix}.\]

Thus 
\[W^\perp:=\row
\begin{pmatrix}
1&-1+i&-1-i&0&0&1\\
1&1&1&0&-3&0\\
1&1&1&-3&0&0
\end{pmatrix}.\]

Let $\MM_W$ be the $\C$-matroid with $\VV^*(M_W)=W$, and thus $\CC(\MM_W)=\minsupp(W^\perp-\{\0\})$. From the second matrix above we can find a representative from each $(\C-\{0\})$ orbit of $\CC(\MM_W)$, as follows.
$$\begin{array}{rccccccl}
(&1,&-1+i,&-1-i,&0,&0,&1&)\\
(&2+i,&2i,&0,&-3-3i,&0,&1&)\\
(&2-i,&0,&-2i,&-3+3i,&0,&1&)\\
(&0,&2-i,&2+i,&-3,&0,&-1&)\\
(&2+i,&2i,&0,&0,&-3-3i,&1&)\\
(&2-i,&0,&-2i,&0,&-3+3i,&1&)\\
(&0,&2-i,&2+i,&0,&-3,&-1&)\\[12pt]

(&1,&1,&1,&0,&-3,&0&)\\
(&1,&1,&1,&-3,&0,&0&)\\
(&0,&0,&0,&1,&-1,&0&)\\
\end{array}$$
Let $\MM'=\ph(\MM_W)$. Thus the phases of the elements of the above list give us a list of representatives for the $(\P-\{0\})$ orbits in $\CC(\MM')$. Since $\CC(\MM'/6)=\minsupp\{X\backslash 6: X\in\CC(\MM)\}$, the phases of the first seven elements of the above list, with their sixth components removed, give us the elements of $\CC(\MM'/6)$. By drawing phase diagrams we check that any element of $(S^1)^5$ sufficiently close to $(1,1,1,1,1)$ is orthogonal to each of these 7 elements of $\CC(\MM'/6)$, hence is in $\VV^*(\MM'/6)$. For instance, if $\beta\in S^1$ is close to but not equal to 1, then $(1,1,1,1,\beta)\in\VV^*(\MM/6)$. But certainly $(1,1,1,1,\beta,0)\not\in\VV^*(\MM')$, since $(1,1,1,1,\beta,0)$ is not orthogonal to the element $(0,0,0,1,-1,0)$ of $\CC(\MM')$.

\section{Flats and Composition}\label{sec:flatcompos}

When we think of oriented matroids as matroids with extra structure, the signed covector set of an oriented matroid can be thought of as extra structure on the lattice of flats of the underlying matroid. That every flat underlies some signed covector follows from the  Composition Axiom for oriented matroids. 

For general $F$-matroids, this falls apart: for instance,  for an $\F_2$-matroid $\MM$, not every flat need arise as the underlying flat of an element of $\VV^*(\MM)$.  Section~\ref{sec:flats} will sketch the relationship between flats and $F$-covectors, and Section~\ref{sec:compos} will propose a general notion of a ``composition operation", defined for a particular tract $F$, so that existence of a composition operation implies the same relationship between flats and $F$-covector sets as we have in the case of oriented matroids. We will then explore composition operations on some particular tracts.

\subsection{Flats}\label{sec:flats}

\begin{defn} A {\bf hyperplane} of a matroid is the complement of a cocircuit.

A {\bf flat} of a matroid is an intersection of hyperplanes.
\end{defn}

As always, when we refer to matroid properties of an $F$-matroid we mean properties of the underlying matroid. 

\begin{prop}\label{prop:flat1} Let $F$ be a tract and $\MM$ an $F$-matroid.
\begin{enumerate}
\item The set of hyperplanes of $\MM$ is $\{X^0: X\in\CC^*(\MM)\}$.
\item $\{X^0:X\in\VV^*(\MM)\}$ is a subset of the set of flats of $\MM$.
\item If $F$ is an infinite field or $F\in\{\K,\S\}$ then $\{X^0:X\in\VV^*(\MM)\}$ is exactly the set of flats of $\MM$.
\item If $F$ is an infinite tract satisfying the Weak Closure Property then $\{X^0:X\in\VV^*(\MM)\}$ is exactly the set of flats of $\MM$.
\item If $F$ is a finite field then there is an $F$-matroid $\MM$ such that $\emptyset$ is a flat and $\emptyset\not\in\{X^0:X\in\VV^*(\MM)\}$.
\end{enumerate}
\end{prop}

\begin{proof} (1) follows from Proposition~\ref{prop:morph}, and (2) follows from Proposition~\ref{prop:pushforward} applied to $\kappa: F\to \K$ and Proposition~\ref{prop:matroidvect}.

If $F=\K$ then by Proposition~\ref{prop:matroidvect} $\{X^0:X\in\VV^*(\MM)\}$ is exactly the set of flats of $\MM$. If $F=\S$ then the Composition Axiom for oriented matroids implies that $\{X^0:X\in\VV^*(\MM)\}$ is exactly the set of flats of $\MM$. 

Since a field is an example of a tract satisfying the Weak Closure Property, we prove the remainder of (3) by proving (4). Let $X, Y\in F^E$. For every $e\in E$ there is at most one $\alpha_e\in F$ such that $X(e)=-\alpha_eY(e)$. Thus there are only finitely many values $\alpha\in F$ such that $X\hplus\alpha Y$ contains an element $Z$ such that $\supp{Z}\neq\supp{X}\cup\supp{Y}$. Proceeding inductively, we see that for any $X_1, \ldots X_k\in F^E$, there are only finitely many values $\alpha_2, \ldots, \alpha_k\in F$ such that $X_1\hplus\bighplus_{j=2}^k\alpha_j X_j$ contains an element $Z$ such that $\supp{Z}\neq\bigcup_{j=1}^k\supp{X_j}$. Any flat has the form $\bigcap_{j=1}^k X_j^0$ for cocircuits $X_1, \ldots, X_k$, so by taking an appropriate linear combination of these $X_j$ we get $Z\in\VV^*(\MM)$ with $Z^0=\bigcap_{j=1}^k X_j^0$.

To prove (5), let $F-\{0\}=\{a_1,\ldots, a_n\}$, and let $\MM$ be the rank 2 $F$-matroid on $E=[n+2]$ with 
$$\VV^*(M)=\row
\begin{pmatrix}
1 & 0 & 1 & 1 & \cdots& 1\\
 0& 1  &a_1 & a_2 &\cdots&a_n 
\end{pmatrix}$$
Then the hyperplanes of $\MM$ are exactly the single-element subsets of $E$, and so the empty set is a flat. However, any element of $\VV^*(\MM)$ (i.e., any linear combination of the two rows) has a 0 coordinate.
\end{proof}

It would be interesting to characterize the tracts for which $\{X^0:X\in\VV^*(\MM)\}$ is not always the set of flats of an $F$-matroid $\MM$, and to better understand the extent to which these tracts ``behave like finite fields".
This is one motivation for looking at {\em composition operations}, the subject of Section~\ref{sec:compos}. Proposition~\ref{prop:compflat} will  prove that if a tract $F$ admits a composition operation then $\{X^0:X\in\VV^*(\MM)\}$ is the set of flats of $\MM$ for every $F$-matroid $\MM$.  We'll then find composition operations for all of the tracts introduced in Example~\ref{ex:tracts} except $\P$. 

We do not know a composition operation for $\P$, and we do not know if every flat of a $\P$-matroid is the 0 set of a $\P$-covector.

\subsection{Composition operations}\label{sec:compos}
The usual system of axioms for signed vectors of  oriented matroids includes a {\em Composition Property}, which says that
 if $X,Y\in \VV(\MM)$  then $X\circ Y\in \S^E$ defined by $$X\circ Y(e)=\begin{cases}
X(e)&\mbox{ if $X(e)\neq 0$}\\
Y(e)&\mbox{  otherwise}
\end{cases}$$
is also in $\VV(\MM)$.

\begin{defn} \label{def:composition} A {\bf composition operation} on a tract $F$ is a hyperoperation $\circ_F$ defined on $F^E$ for all finite $E$ such that
\begin{enumerate}
\item For every $X_1$ and $ X_2$ and every $Y\in  X_1\circ_F  X_2$, $\supp{Y}=\supp{ X_1}\cup\supp{ X_2}$, and
\item If $X_1\in Z^\perp$ and $ X_2\in Z^\perp $ then  $X_1\circ_F  X_2\subseteq Z^\perp$.
\end{enumerate}
\end{defn}
In other words, a composition operation associates to each $X,Y\in F^E$ a nonempty subset of $(X^\perp\cap Y^\perp)^\perp\cap\{Z: \supp{Z}=\supp{X}\cup\supp{Y}\}$.

When $| X_1\circ_F X_2|=1$ for all $X_1$ and $X_2$, we will often treat $\circ_F$ as a binary operation -- as we already do for the usual composition for oriented matroids.

\begin{remark} \label{rem: compos} This definition is chosen to give us the hypotheses needed to prove the results of this section. All of the examples we will consider, including ordinary composition of oriented matroids, are, in addition, associative. Further, all of them except $\epsilon$-composition and its inspiration, Example~\ref{ex:realcomp},  satisfy the condition that $X\circ_F Y\subseteq X\hplus Y$ for all $X$ and $Y$.
 Both of these additional properties align
with the geometric motivation for composition from oriented matroids,  and both come up frequently in oriented matroid proofs. Thus Definition~\ref{def:composition}  should not necessarily be taken as, well, definitive.
\end{remark}

\begin{example} \label{ex:realcomp} The geometric motivation behind the definition of composition for oriented matroids is the observation that, if $X,Y\in\R^E$, then for all sufficiently small $\epsilon>0$, $\sign(X+\epsilon Y)=\sign(X)\circ\sign(Y)$, where the  composition on the right-hand side is the usual oriented matroid composition.

Specifically, given $X,Y\in \R^E$, let $$\epsilon_0=\min\left(\left|\frac{X(e)}{Y(e)}\right|:  X(e)Y(e)< 0\right).$$ 
(If there is no $e$ such that $X(e)Y(e)< 0$ then let $\epsilon_0=\infty$.) Then 
$X\circ_\R Y:=\{X+\epsilon Y: \epsilon<\epsilon_0\}$ is a composition operation, with the additional property that, for every $Z\in X\circ_\R Y$,  $\sign (Z)=\sign(X)\circ \sign(Y)$.
\end{example}

\begin{defn} 1. For $X_1, \ldots, X_k\in F^E$, define $X_1\circ_F\cdots \circ_F X_k$ recursively: 
$$X_1\circ_F\cdots \circ_F X_k=\bigcup_{Y\in X_1\circ_F\cdots \circ_F X_{k-1}} Y\circ_F X_k.$$

2. For ${\mathcal S}\subseteq F^E$, let ${\mathcal S}_\circ$ denote the union of all compositions  $X_1\circ_F\cdots\circ_F X_k$, where $k\in\mathbb N$ and $X_1, \ldots X_k\in {\mathcal S}$.
\end{defn}

\begin{prop} \label{prop:compflat} Let  $F$ be a tract admitting a composition operation $\circ_F$, and let $\MM$ be an $F$-matroid.
\begin{enumerate}
\item If $X, Y\in \VV(\MM)$ then $X\circ_F Y\subseteq\VV(\MM)$.
\item $(\CC^*(\MM)_\circ)^\perp= \VV(\MM)$.
\item  $\{X^0:X\in\VV^*(\MM)\}$ is exactly the set of flats of the underlying matriod of $\MM$.
\end{enumerate}
\end{prop}

\begin{proof} The first two statements follow immediately from $\VV(\MM)=\CC^*(\MM)^\perp$.

By Proposition~\ref{prop:matroidvect}  $A\subseteq E$ 
is a flat if and only if $E-A$ is a union of cocircuits of the underlying matroid. For any set of cocircuits, Corollary~\ref{cor:underlying} promises the existence of $F$-cocircuits of $\MM$ with these supports. The composition of these cocircuits is a subset of  $\VV^*(\MM)$ consisting of elements with support $E-A$. Thus every flat of $\MM$ is $X^0$ for some $X\in\VV^*(\MM)$. The converse is given by Proposition~\ref{prop:flat1}.
\end{proof}

\begin{example} \label{ex:OMcomp} Every signed vector of an oriented matroid is a composition of signed circuits (cf. 3.7.2 in~\cite{BLSWZ}). Thus if $F=\mathbb S$ then $\VV(\MM)=\CC(\MM)_\circ$, and so 
$\VV(\MM)^\perp=\VV^*(\MM)$. (Compare to Section~\ref{sec:duality}.)
\end{example}

It would be interesting to characterize the tracts $F$ with some notion of composition (possibly with the additional constraints of Remark~\ref{rem: compos}) such that every $F$-vector of an $F$-matroid $\MM$ is contained in a composition of $F$-circuits of $\MM$.
 By Proposition~\ref{prop:compflat}.2, such tracts are perfect.

\subsubsection{The Inflation Property}\label{sec:absorp}

\begin{defn} A tract is said to have the {\bf Inflation Property} if, whenever $\sum_{j=1}^k a_j\in N_G-\{0\}$ we have $b+\sum_{j=1}^k a_j\in N_G$ for all $b\in G$.
\end{defn}

\begin{prop}\label{prop:inflate} For any tract $F$, the following are equivalent.
\begin{enumerate}
\item If $a\in F-\{0\}$ then $a\hplus -a=F$.
\item $1\hplus -1=F$.
\item For every $a\in F-\{0\}$ and $b\in F$, $a\in a\hplus b$.
\end{enumerate}
Further any hyperfield having these properties satisfies the Inflation Property.
\end{prop}

\begin{proof}  (1) is equivalent to (2) because $N_G$ is invariant under multiplication by elements of $G$.

For every $a, b\in F$, $b\in a\hplus -a$ if and only if $a\in a\hplus b$. Thus (1) is equivalent to (3).

If $F$ is a hyperfield satisfying these properties, $a_1\in F-\{0\}$, and $\sum_{j=1}^k a_j\in N_G$ then $-a_1\in \bighplus_{j=2}^k a_j$, and so $\bighplus_{j=1}^k a_j=(a_1\hplus -a_1)\cup(a_1\hplus(\bighplus_{j=2}^k a_j-\{-a_1\})=F$.
\end{proof}

\begin{example}  $\K$, $\mathbb S$, and $\TP$ all satisfy the Inflation Property. 
.
\end{example}

\begin{prop} If $F$ satisfies the Inflation Property then the composition $\circ$ defined by
$$X\circ Y(e)=\begin{cases}
X(e)&\mbox{ if $X(e)\neq 0$}\\
Y(e)&\mbox{  otherwise}
\end{cases}$$
is a composition operation.
\end{prop}

\begin{proof} Let $X,Y,Z\in F^E$ such that $X\perp Z$ and $Y\perp Z$.

If $X\htimes Z=\{0\}$, then $(X\circ Y)\htimes Z=Y\htimes Z$, thus $X\circ Y\perp Z$.

Otherwise, $$(X\circ Y)\htimes Z=\sum_{j\in\supp{X}} X(j)Z(j)^c + \sum_{j\in X^0}Y(j)Z(j)^c$$ and since $\sum_{j\in\supp{X}} X(j)Z(j)^c\in N_G$ we have $(X\circ Y)\htimes Z\in N_G$.
\end{proof}

\subsubsection{$\T$-composition}

Throughout the remaining sections $+$ and $\sum$ denote ordinary addition in $\R$.

\begin{prop} \label{prop:t-composition}
Let $F\in\{\triangle, \TP, \T\R, \T\C, \T\triangle\}$, and let $\MM$ be an $F$-matroid. Then the operation
$$X\circ_{\mmax} Y(e)=\begin{cases}
X(e)&\mbox{ if $|X(e)|\geq |Y(e)|$}\\
Y(e)&\mbox{  otherwise.}
\end{cases}$$
is a composition operation on $F$.
\end{prop}

\begin{lemma} \label{lem:trianglesum} Let $\{s_1, \ldots, s_k\}$ be a finite subset of $\triangle$  with $s_1\leq\cdots\leq s_k$. Then $0\in\bighplus_{j=1}^k s_j$ if and only if $s_k\leq\sum_{j=1}^{k-1} s_j$.
\end{lemma}

\begin{proof} $0\in(\bighplus_{j=1}^{k-1} s_j)\hplus s_k$ if and only if $s_k\in\bighplus_{j=1}^{k-1} s_j$. Induction on $k$ shows that the smallest element of $\bighplus_{j=1}^{k-1} s_j$ is $\max(0, s_{k-1}-s_{k-1}-\cdots-s_1)$, while the largest element of $\bighplus_{j=1}^{k-1} s_j$ is $\sum_{j=1}^{k-1} s_j$. Thus 
$$s_k\in\bighplus_{j=1}^{k-1} s_j\Leftrightarrow \max(0, s_{k-1}-s_{k-1}-\cdots-s_1)\leq s_k\leq \sum_{j=1}^{k-1} s_j$$
But the first inequality above is vacuous: for any $s_k\in \triangle$ we have $0\leq s_k$, and by hypothesis $s_k\geq s_{k-1}\geq s_{k-1}-s_{k-1}-\cdots-s_1$.

\end{proof}

\begin{lemma}\label{lem:tropadd}  Let $F\in\{\TP, \T\R, \T\C, \T\triangle\}$, and let $(s_j:j\in J)$ be a sequence in $F$. Then $0\in\bighplus_{j\in J} s_j$ if and only if there is a $J'\subseteq J$ so that 
\begin{enumerate}
\item $|s_j|\leq|s_{j'}|$ for all $j\in J$ and $j'\in J'$ and 
\item $0\in\bighplus_{j\in J'} s_j$.
\end{enumerate}

Also,
\begin{enumerate}
\item if $F=\TP$ or $F=\T\C$ then $0\in\bighplus_{j\in J'} s_j$ if and only if $\{s_j: j\in J'\}$ is not contained in an open half circle,
\item if $F=\T\R$  then $0\in\bighplus_{j\in J'} s_j$ if and only if  there exists $j_1, j_2\in J'$ such that $s_{j_1}=-s_{j_2}$, and
\item if $F=\T\triangle$
 then $0\in\bighplus_{j\in J'} s_j$ if and only if either $s_j=0$ for all $j$ or  there exists $j_1\neq j_2\in J'$ such that $s_{j_1}=s_{j_2}$.
 \end{enumerate}
\end{lemma}

The proof is easy.

Now we can prove Proposition~\ref{prop:t-composition}:
\begin{proof}  
We first give the argument for $F=\triangle$. Let $X,Y,Z\in \triangle^E$ such that $X\perp Z$ and $Y\perp Z$.  Consider an $f$ with $|(X\circ_{\mmax} Y)(f)Z(f)^c|$ as large as possible. Then
\begin{align*}
\sum_{e\neq f} (X\circ_{\mmax} Y)(e)Z(e)^c&\geq \sum_{e\neq f} X(e)Z(e)^c\\[12pt]
&\geq X(f)Z(f)^c\qquad\mbox{ since $0\in X\htimes Z$, by Lemma~\ref{lem:trianglesum}}
\end{align*}
and similarly $\sum_{e\neq f} (X\circ_{\mmax} Y)(e)Z(e)^c\geq Y(f)Z(f)^c$. Since $(X\circ_{\mmax} Y)(f)\in\{X(f), Y(f)\}$, the result follows.

For $F\in \{\TP, \T\R, \T\C, \T\triangle\}$,
let $X,Y,Z\in F^E$ such that $X\perp Z$ and $Y\perp Z$. Consider an $f$ with $|(X\circ_{\mmax} Y)(f)Z(f)^c|$ as large as possible. Thus for every $e$
\begin{align*}
|(X\circ_{\mmax} Y)(f)Z(f)^c|&\geq|(X\circ_{\mmax} Y)(e)Z(e)^c|\\
&=|(X\circ_{\mmax} Y)(e)||Z(e)^c|\\
&\geq|X(e)Z(e)^c|.
\end{align*}
Similarly $|(X\circ_{\mmax} Y)(f)Z(f)^c|\geq |Y(e)Z(e)^c|$. 

If there is such an $f$ with $(X\circ Y)(f)=X(f)$, then orthogonality of $X$ and $Z$  implies the existence of $E_0\subseteq E$ with
$|X(e)Z(e)^c|=|X(f)Z(f)^c|$ for all $e\in E_0$, and $0\in\bighplus_{e\in E_0} X(e)Z(e)^c$.
Since for $e\in E_0$
\begin{align*}
|X(e)Z(e)^c|&=|(X\circ_T Y)(f)Z(f)^c|\\
&\geq|Y(e)Z(e)^c|
\end{align*}
we have $|X(e)|\geq |Y(e)|$, and so $(X\circ_T Y)(e)=X(e)$. Thus for all $e'\in E$ and $e\in E_0$, $$|(X\circ_T Y)(e')Z(e')^c|\leq |(X\circ_T Y)(e)Z(e)^c|$$ and $$0\in\bighplus_{e\in E_0} (X\circ_T Y)(e)Z(e)^c$$ so $0\in (X\circ_T Y)\htimes Z$.

A similar argument covers the case when $(X\circ Y)(f)=Y(f)$ for all such $f$.
\end{proof}

\subsubsection{ $\epsilon$-composition}

This section gives a $\T$-analog to the operation of Example~\ref{ex:realcomp}.

\begin{defn} Let $F\in \{\triangle,\TR,\TC,\T\triangle\}$, and let $X,Y\in F^E$.
For each real number $\epsilon>0$, define $X\circ_\epsilon Y\in F^E$ to be the set of $Z$ such that, for some $\omega$ with $0<\omega<\epsilon$:
$$Z(e)=\begin{cases}
X(e)&\mbox{ if $X(e)\neq 0$}\\
\omega Y(e)&\mbox{ if $X(e)= 0$.}\end{cases}$$
\end{defn}

\begin{prop}  Let $F\in \{\triangle,\TR,\TC,\T\triangle\}$, let $\MM$ be an $F$-matroid, and let $X,Y\in\VV^*(\MM)$. Then there is an $\epsilon>0$ such that $X\circ _{\epsilon}Y\subseteq\VV^*(\MM)$.
\end{prop}

\begin{proof} Let $\{Z_1, \ldots, Z_k\}$ be a choice of one $F$-circuit from each $G$-orbit of $\CC(\MM)$. For each $j$ such that $\supp{X}\cap\supp{Z_j}\neq\emptyset$,  let $ d_j=\max(|X(e)||Z_j(e)|:e\in E)$, and let  $\epsilon_j=\min(|\frac{ d_j}{Z_j(f)}|:f\in\supp{Z}-\supp{X})$. Then let $\epsilon$ be the minimum over all $\epsilon_j$. 

For all $j$ such that $\supp{X}\cap\supp{Z_j}=\emptyset$ and for all $X'\in X\circ_\epsilon Y$ we have $X'\htimes Z_j=\omega Y\htimes Z_j$, for some $\omega$, and so $X'\perp Z_j$.

If $\supp{X}\cap\supp{Z_j}\neq\emptyset$, we have slightly different arguments for $F=\triangle$ and for $F\in\{\TR,\TC,\T\triangle\}$.

If $F=\triangle$ and $X'\in X\circ_\epsilon Y$, consider an $f$ with $|X'(f)Z_j(f)^c|$ as large as possible. 
Then
\begin{align*}
\sum_{e\neq f} X'(e)Z_j(e)^c&\geq \sum_{e\neq f} X(e)Z_j(e)^c\\[12pt]
&\geq X(f)Z_j(f)^c\qquad\mbox{ since $0\in X\htimes Z_j$, by Lemma~\ref{lem:trianglesum}}
\end{align*}
Also 
\begin{align*}
\sum_{e\neq f} X'(e)Z_j(e)^c&\geq \sum_{e\neq f} \omega Y(e)Z_j(e)^c\\[12pt]
&\geq\omega Y(f)Z_j(f)^c\qquad\mbox{ since $0\in Y\htimes Z_j$.}
\end{align*}
 Since $X'(f)\in\{X(f), \omega Y(f)\}$, the result follows.

If $F\in \{\TR,\TC,\T\triangle\}$, then $X\htimes Z_j=\bighplus_{e\in\supp{X}} X(e)Z(e)^c=\{\alpha\in F: |\alpha|\leq  d_j\}$.  Denote this set $I_{ d_j}$. Thus
 for every $X'\in X\circ_\omega Y$ 
 we have
\begin{align*}
X'\htimes Z_j&=\bighplus_{e\in\supp{X}} X(e)Z_j(e)^c \hplus\bighplus_{f\in\supp{Z_j}-\supp{X}}\omega Y(f)Z_j(f)^c\\
&=I_{ d_j}\hplus\bighplus_{f\in\supp{Z_j}-\supp{X}}\omega Y(f)Z_j(f)^c\\
&=I_{ d_j}.
\end{align*}

Thus $X'\in\CC(\MM)^\perp$.
\end{proof}

\section{Sum properties}\label{sec:elim}

When $K$ is a field and $\MM$ is a $K$-matroid, then $\VV^*(\MM)$ is a vector space, hence is closed under (hyper)addition. In contrast, even for an $\S$-matroid $\MM$, if $X,Y\in\VV^*(\MM)$ then $X\hplus Y$ is not necessarily a subset of $\VV^*(\MM)$. However, matroids over fields, $\K$, and $\S$ all satisfy the following weaker versions of additive closure.

\begin{wcproperty} If $X,Y\in \VV^*(\MM)$ then $(X\hplus Y)\cap \VV^*(\MM)\neq \emptyset$.
\end{wcproperty}

\begin{eproperty}If $X,Y\in \VV^*(\MM)$ and $X(e)=-Y(e)$ then there is a $Z\in (X\hplus Y)\cap \VV^*(\MM)\neq \emptyset$
such that $Z(e)=0$.
\end{eproperty}

\begin{acproperty} If $X,Y\in \VV^*(\MM)$ and $\alpha\in X(e)\hplus Y(e)$ then there is a $Z\in (X\hplus Y)\cap \VV^*(\MM)\neq \emptyset$
such that $Z(e)=\alpha$.
\end{acproperty}

It would be interesting to characterize the tracts  whose matroids satisfy each of these properties.

If $F$ is a tract and $x,y\in F$ with $x\hplus y=\emptyset$, then a rank 1 $F$-matroid with 1 element will not satisfy the Weak Closure Property. Examples of such tracts are given in~\cite{BB17}.

\begin{conj} The Weak Closure Property holds for all matroids over hyperfields.
\end{conj}

Chris Eppolito has recently found an example of a matroid over a hyperfield violating the Elimination Property. 

As a small contribution to the study of $X\hplus Y$, we have the following.

\begin{prop} \label{prop:tropicalclosure}  Let $F\in\{\TR, \TC, \T\triangle\}$. Let $\MM$ be an $F$-matroid and $X_1, X_2\in\VV^*(\MM)$. Assume there is at most one value $e$ such that $|X_1(e)|=|X_2(e)|\neq 0$ and $X_1(e)\neq X_2(e)$. Then $X_1\hplus X_2\subseteq \VV^*(\MM)$.
\end{prop}

\begin{proof} Let $Z\in\CC(\MM)$. If $\supp{Z}\cap\supp{X_1}=\emptyset$ then for every $X'\in X_1\hplus X_2$ we have $X'\htimes Z=X_2\htimes Z$, and so $X'\perp Z$. Similarly if $\supp{Z}\cap\supp{X_2}=\emptyset$ then every $X'\in X_1\hplus X_2$ is in $Z^\perp$.

Now assume $\supp{Z}\cap\supp{X_1}\neq\emptyset$ and $\supp{Z}\cap\supp{X_2}\neq\emptyset$. Thus, by Lemma~\ref{lem:tropadd} there are constants $c_1$, $c_2$ and nonempty subsets $S_1$, $S_2$ of $E$ such that for $j\in\{1,2\}$,
\begin{enumerate}
\item for all $e\in E$, $|X_j(e)Z(e)^c|\leq c_j$, with equality if and only if $e\in S_j$, and
\item $\0\in\bighplus_{e\in S_j} X_j(e)Z(e)^c$.
\end{enumerate}

If $c_1\neq c_2$, without loss of generality assume $c_2>c_1$. Then for each $f_1\in S_2$ and each $X'\in X_1\hplus X_2$ we have $X'(f_1)=X_2(f_1)$. Also, for every $f_2\in E- S_2$ we have $|X'(f_2) Z(f_2)^c|<c_2$. Thus
\begin{align*}
X'\htimes Z&= \sum_{e\in E} X'(e)Z(e)^c\\
&= \sum_{e\in S_2} X'(e)Z(e)^c\\
&= \sum_{e\in S_2} X_2(e)Z(e)^c\in N_G.
\end{align*}

If $c_1=c_2$, then for each $X'\in X_1\hplus X_2$, $|X'(e)Z(e)^c|$ is maximized at each $e\in S_1\cup S_2$, and further, by our hypothesis, $X'(e)=X_1(e)=X_2(e)$ for all but at most one element $e_0$ of $S_1\cap S_2$. 

Thus
\begin{align*}
X'\htimes Z&= \sum_{e\in E} X'(e)Z(e)^c\\
&= \sum_{e\in S_1\cup S_2} X'(e)Z(e)^c
\end{align*}
and for $j\in\{1,2\}$
\begin{align*}
 X_j(e_0)Z(e_0)^c+\sum_{e\in S_1\cup S_2-\{e_0\}} X_j(e)Z(e)^c\in N_G\\[12pt]
-X_j(e_0)Z(e_0)^c\in\bighplus_{e\in S_1\cup S_2-\{e_0\}} X_j(e)Z(e)^c
=\bighplus_{e\in E-\{e_0\}} X'(e)Z(e)^c
\end{align*}

and so by Lemma~\ref{lem:tropclos} 
\begin{align*}
-X'(e_0)Z(e_0)^c&\in\bighplus_{e\in E-\{e_0\}} X'(e)Z(e)^c\\
\sum_{e\in E} X'(e)Z(e)^c&\in N_G\qedhere
\end{align*}
\end{proof}

\begin{lemma}\label{lem:tropclos} $F\in\{\TR, \TC, \T\triangle\}$. Let $S$ be a subset of $F$ and $a,b\in F$. If $\{a,b\}\subseteq \bighplus_{s\in S} s$ then $a\hplus b\subseteq \bighplus_{s\in S} s$.
\end{lemma}

This follows immediately from Lemma~\ref{lem:tropadd}.

\noindent{\bf Acknowledgements.} Many thanks to Matt Baker, Emanuele Delucchi, and Thomas Zaslavsky  for helpful discussions. Chris Eppolito contributed the idea of the Inflation Property. Thanks also to Ting Su for Proposition 4.5 and for many comments and corrections, and to an anonymous referee whose careful reading led to important corrections.

\bibliographystyle{amsalpha}
\bibliography{bibliV}

\newcommand{\etalchar}[1]{$^{#1}$}
\providecommand{\bysame}{\leavevmode\hbox to3em{\hrulefill}\thinspace}
\providecommand{\MR}{\relax\ifhmode\unskip\space\fi MR }
\providecommand{\MRhref}[2]{%
  \href{http://www.ams.org/mathscinet-getitem?mr=#1}{#2}
}
\providecommand{\href}[2]{#2}
\begin{thebibliography}{BLVS{\etalchar{+}}99}

\bibitem[AD12]{AD}
Laura Anderson and Emanuele Delucchi, \emph{Foundations for a theory of complex
  matroids}, Discrete Comput. Geom. \textbf{48} (2012), no.~4, 807--846.
  \MR{3000567}

\bibitem[AD17]{ADavis}
Laura Anderson and James~F. Davis, \emph{Hyperfield {G}rassmannians}, 2017,
  arXiv:1710.00016.

\bibitem[BB16]{BB16}
M.~{Baker} and N.~{Bowler}, \emph{{Matroids over hyperfields}}, 2016,
  arXiv:1601.01204.

\bibitem[BB17]{BB17}
\bysame, \emph{{Matroids over partial hyperstructures}}, 2017,
  arXiv:1709.09707.

\bibitem[BLVS{\etalchar{+}}99]{BLSWZ}
Anders Bj{\"o}rner, Michel Las~Vergnas, Bernd Sturmfels, Neil White, and
  G{\"u}nter~M. Ziegler, \emph{Oriented matroids}, second ed., Encyclopedia of
  Mathematics and its Applications, vol.~46, Cambridge University Press,
  Cambridge, 1999. \MR{MR1744046 (2000j:52016)}

\bibitem[Del11]{Del11}
Emanuele Delucchi, \emph{Modular elimination in matroids and oriented
  matroids}, European J. Combin. \textbf{32} (2011), no.~3, 339--343.
  \MR{2764796}

\bibitem[DW92a]{DWperfect}
Andreas W.~M. Dress and Walter Wenzel, \emph{Perfect matroids}, Adv. Math.
  \textbf{91} (1992), no.~2, 158--208. \MR{1149622}

\bibitem[DW92b]{DW92}
\bysame, \emph{Valuated matroids}, Adv. Math. \textbf{93} (1992), no.~2,
  214--250. \MR{1164708}

\bibitem[GG15]{gian}
Jeffrey Giansiracusa and Noah Giansiracusa, \emph{A {G}rassmann algebra for
  matroids}, 2015, arXiv:1510.04584.

\bibitem[Jun17]{jun}
Jaiung Jun, \emph{Geometry of hyperfields}, 2017, arXiv:1707.09348.

\bibitem[Oxl92]{Oxley}
James~G. Oxley, \emph{Matroid theory}, Oxford Science Publications, The
  Clarendon Press Oxford University Press, New York, 1992. \MR{MR1207587
  (94d:05033)}

\bibitem[Vir10]{Viro1}
Oleg Viro, \emph{Hyperfields for tropical geometry {I}. hyperfields and
  dequantization}, 2010, arXiv:1006.3034.

\end{thebibliography}

\end{document}